\documentclass[a4paper,11pt]{article}

\usepackage{amsthm}
\usepackage{amsmath}
\usepackage{amssymb}
\usepackage{graphicx}

\usepackage{authblk}  
\usepackage{caption}
\usepackage{subcaption} 

\usepackage{xcolor} 

\theoremstyle{plain}
\newtheorem{theorem}{\textbf{Theorem}}
\newtheorem{proposition}[theorem]{\textbf{Proposition}}
\newtheorem{claim}[theorem]{\textbf{Claim}}
\newtheorem{lemma}[theorem]{\textbf{Lemma}}
\newtheorem{conjecture}[theorem]{\textbf{Conjecture}}
\newtheorem{property}[theorem]{\textbf{Property}}
\newtheorem{corollary}[theorem]{\textbf{Corollary}}
\newtheorem*{proposition*}{\textbf{Proposition}}
\newtheorem*{theorem*}{\textbf{Theorem}}

\newcommand{\ZZ}{\ensuremath{\mathbb{Z}}}
\newcommand{\NN}{\ensuremath{\mathbb{N}}}
\newcommand{\RR}{\ensuremath{\mathbb{R}}}
\newcommand{\QQ}{\ensuremath{\mathbb{Q}}}
\newcommand{\setw}{\ensuremath{\{a,b\}^\star}}
\newcommand{\bw}{\ensuremath{\mathbf{w}}}
\newcommand{\bp}{\ensuremath{\mathbf{p}}}

\newcommand{\bv}{{\ensuremath{\mathbf{v}}}}
\newcommand{\bu}{{\ensuremath{\mathbf{u}}}}
\newcommand{\bz}{{\ensuremath{\mathbf{z}}}}
\newcommand{\beps}{\ensuremath{\boldsymbol{\varepsilon}}}

\newcommand{\defeq}{\overset{\mathrm{def}}{=}}
\newcommand{\chris}[1]{\ensuremath{\mathfrak{c}_{#1}}}
\newcommand{\bel}{\ensuremath{\mathrm{Bel}}}
\newcommand{\bord}{\ensuremath{\mathrm{Hull}}}
\newcommand{\hull}{\ensuremath{\mathrm{Hull}}}
\renewcommand{\flat}{\ensuremath{\mathrm{Flat}}}
\newcommand{\lift}{\ensuremath{\mathrm{Lift}}}
\newcommand{\Fibo}{\ensuremath{\mathrm{F}}}
\newcommand{\Pell}{\ensuremath{\mathrm{P}}}

\DeclareMathOperator{\mvalue}{\ensuremath{m-\textrm{value}}}

\title{On the Markov numbers: fixed numerator, denominator, and sum conjectures}
\author[1]{Clément Lagisquet}
\author[2]{Edita Pelantová}
\author[1]{Sébastien Tavenas}
\author[1]{Laurent Vuillon}
\affil[1]{LAMA, Université Savoie Mont-Blanc, CNRS}
\affil[2]{Faculty of Nuclear Sciences and Physical Engineering\\ Czech Technical University in Prague}

\begin{document}
\maketitle 

\begin{abstract}
	The Markov numbers are the positive integer solutions of the Diophantine equation \(x^2 + y^2 + z^2 = 3xyz\). Already in 1880, Markov showed that all these solutions could be generated along a binary tree. So it became quite usual (and useful) to index the Markov numbers by the rationals from \([0,1]\) which stand at the same place in the Stern–Brocot binary tree. 
	The Frobenius' conjecture claims that each Markov number appears at most once in the tree.
	
	In particular, if the conjecture is true, the order of Markov numbers would establish a new strict order on the rationals. Aigner suggested three conjectures to better understand this order. The first one has already been solved for a few months. We prove that the other two conjectures are also true.
	
	Along the way, we generalize Markov numbers to any couple \((p,q)\in \NN^2\) (not only when they are relatively primes) and conjecture that the unicity is still true as soon as \(p\leq q\). Finally, we show that the three conjectures are in fact true for this superset.
\end{abstract}

\subsection*{Acknowledgement:}
The research received funding from the Ministry of Education, Youth and Sports of the Czech Republic through the project \(\text{n}^{\circ}\) CZ.02.1.01/0.0/0.0/16\_019/0000778 and from the project BIRCA from the IDEX Université Grenoble Alpes.

\section{Introduction}

Markov numbers have been introduced in~\cite{Ma79, Mar80} to describe minimal values of quadratic forms on integer points, thus linking this question to the topic of Diophantine equations. Since then, connections with this sequence of numbers have emerged in many areas.
The book of Aigner \cite{Aig15} about the 100 years of the Frobenius' conjecture on Markov numbers  is based on many techniques coming from number theory \cite{Bo07}, hyperbolic geometry ~\cite{Ser85, Co55}, matching theory \cite{Propp20} and combinatorics on words \cite{RV16} and  for each technique a catalogue of interesting conjectures are stated.

The Frobenius' conjecture~\cite{Fro13} is linked with the solutions of the Diophantine equation $x^2+y^2+z^2=3xyz$ where each solution with non negative integers is called a Markov triple. In fact, except $(1,1,1)$ and $(1,2,1)$ all other Markov triples contain pairwise distinct integers. In addition, if a Markov triple $(x,y,z)$ with pairwise distinct values satisfies the Diophantine equation and has  $y$ as a maximum of the triple, then this triple gives birth to two other solutions, which are $(x,3xy-z,y)$ and $(y,3yz-x,z).$
This implies that the triples with pairwise distinct values could be described by a binary tree called Markov tree. The Markov numbers are exactly the numbers which appear in these triples. In 1913, Frobenius conjectured that these numbers appear once and exactly once as the maximum of a triple.

It is usual to define the Farey tree to construct the set of all rational numbers between 0 and 1 that is all $p/q$ with $p$ and $q$ relatively prime positive integers. This tree is defined using the Farey rules: a node $(\frac{a}{b}, \frac{a+c}{b+d},\frac{c}{d})$ generates the two next triples $(\frac{a}{b},\frac{a+(a+c)}{b+(b+d)}, \frac{a+c}{b+d})$ and $(\frac{a+c}{b+d},\frac{(a+c)+c}{(b+d)+d},\frac{c}{d})$ (the tree with only the maximums of these triplets is also known as the Stern-Brocot tree).

In particular, we can index all Markov numbers by the Farey fraction which stands at the same place in the Stern-Brocot tree (see \cite{Aig15, BLRS09}), this correspondence would be one to one if the unicity conjecture was true. We refer to a Markov number by its associated rational in the Farey tree namely $m_\frac{p}{q}$ (following the notation in \cite{RS20}).

It would be very helpful to understand the growth of Markov numbers according to their indices.
In the book of Aigner, we find three nice conjectures. 
\begin{enumerate}
	\item (The fixed numerator conjecture) Let $p,q$ and $i$ in \(\NN\) such that \(i>0\), $p<q, \gcd(p,q)=1$ and $\gcd(p,q+i) = 1$ then $m_\frac{p}{q} <m_\frac{p}{q+i}.$ 
	\item (The fixed denominator conjecture) Let $p,q$ and $i$ in \(\NN\) such that \(i>0\), $p+i<q, \gcd(p,q)=1$ and $\gcd(p+i,q) = 1$ then $m_\frac{p}{q} <m_\frac{p+i}{q}$.
	\item (The fixed sum conjecture) Let $p,q$ and $i$ be positive integers such that $i<p<q, \gcd(p,q)=1$ and $\gcd(p-i,q+i)=1$ then $m_\frac{p}{q} <m_\frac{p-i}{q+i}$.
\end{enumerate}

The fixed numerator conjecture was proved in \cite{RS20} using cluster algebra and snake graphs. The proof is quite technical with many interesting tools on discrete paths on the $\NN^2$ grid. Nevertheless, one key concept is to use perfect matching on snake graphs (see \cite{Propp20, RS20}) and while the technique is fine this add too much concepts on the transformation of paths.

In this article, we prove directly the fixed numerator conjecture by using transformations on paths on the $\NN^2$ grid. More precisely, to any path in the \(\NN^2\) grid, we associate an integer value (we call it the \(m\)-value) by substituting any horizontal step by the matrix \(A \defeq \begin{pmatrix} 1 & 1 \\ 1 & 2 \end{pmatrix}\) and any vertical step by \(B \defeq \begin{pmatrix} 2 & 1 \\ 1 & 1 \end{pmatrix}\), the \(m\)-value is the top-right entry of the product of these matrices. We will show (Proposition~\ref{prop:cohnMarkov} and~\ref{prop:minimality}) that 
\begin{theorem}
	if  \(p<q\) are relatively prime, then
	\[m_{p/q} = \min_{\bw\textrm{path from }(0,0)\textrm{ to }(q,p)} (\mvalue(\bw)).
	\]
\end{theorem}

Removing the relatively prime constraint, we easily extend the set of Markov numbers to all pairs $(q,p)$ of integers on the $\mathbb{N}^2$ grid: \(m_{p,q} = \min_{\bw\textrm{path from }(0,0)\textrm{ to }(q,p)} (\mvalue(\bw))\). Interestingly, this superset contains all Fibonacci and Pell numbers (and not only the odd-indexed ones). Then, we conjecture these numbers are still unique as soon as \(p \leq q\). This stronger unicity conjecture was verified by computation for the values $ 0 \leq p \leq q \leq 1000$.
Now the three previous conjectures on fixed parameters directly follow from our main result (a combination of Propositions~\ref{prop:FDC},~\ref{prop:FSC} and Corollary~\ref{prop:FNC}).
\begin{theorem}\label{thm:main}
		Let \(p<q \in \NN\). We have 
		\begin{enumerate}
			\item \(m_{p,q} < m_{p,q+1}\),
			\item \(m_{p,q} < m_{p+1,q}\),
			\item and \(m_{p+1,q} < m_{p,q+1}\).
		\end{enumerate}
\end{theorem}

In Proposition~\ref{prop:minimality}, we will see that the minimum \(m\)-value among all paths ending in a same point is attained when the path is a Christoffel word (or a power of a Christoffel word). If \(q,p\) are relatively prime, the Christoffel word is the path in \(\NN^2\) which is the ``closest" of the line linking \((q,p)\) with the origin. One of the main tool is the extension of these Christoffel words to their repetitions when \(q,p\) are not anymore relatively prime. 

Finally, we need to consider all paths and not only powers of Christoffel ones. Indeed, we will transform a power of Christoffel into another one via a sequence of local flips. And in particular, the intermediate paths are not anymore powers of a Christoffel word.

It seems that one originality in our work is the choice of our couple of matrices \(A\) and \(B\). Even if this couple is well known (they are the standard generators for the commutator subgroup of \(\rm{Sl}_2(\ZZ)\)), most of the previous work is based on the couple \(\begin{pmatrix} 1&1 \\ 1&2\end{pmatrix}\) and \(\begin{pmatrix} 3&2 \\ 4&3\end{pmatrix}\). There is a natural reason to that. The standard first triplet of Cohn matrices is \(\left(\begin{pmatrix} 1&1 \\ 1&2\end{pmatrix}, \begin{pmatrix} 7&5 \\ 11&8\end{pmatrix}, \begin{pmatrix} 3&2 \\ 4&3\end{pmatrix}\right)\). People usually identify this triplet with the triplet of words \((\mathfrak{a},\mathfrak{ab},\mathfrak{b})\) which explains the last couple. However, in the Farey tree, the first triplet is \((\frac{0}{1},\frac{1}{2},\frac{1}{1})\), so it becomes more convenient for us to work with the words \((a,aab,ab)\) since like that the numerator and the denominator of the Farey fraction directly give the number of `b' and of `a' in our words. Particularly, identifying the word \(ab\) with  \(\begin{pmatrix} 3&2 \\ 4&3\end{pmatrix}\), we get our choice for \(A\) and \(B\).

In Appendix~\ref{app:CohnW} we focus on the translation between both models showing that Proposition~\ref{prop:cohnMarkov} is really a rewriting of known results. Then, in Aigner's book, after stating these three conjectures the author describes in few pages some ideas to attack these conjectures. We show in Appendix~\ref{app:gFunc} how our work relates to this approach solving on the way some of the questions which are asked there.

The paper is organized as follows. In the first section, we define the tools and objects we will need in the paper. We start with basic tools as the words, paths on the $\NN^2$ grid and the Christoffel words. In the same time, we define the \(m\)-value. Then, we are interested in some initial properties of this \(m\)-value. In particular, we will focus (Lemmas~\ref{lem:simpleFlip} and~\ref{lem:compositeFlip}) on the variation of this value during a flip transformation. We ends this first section with some basic facts on integer geometry. The second section focuses on the fixed parameters conjectures. We start with the study of a path transformation: the flattening. We show that the \(m\)-value strictly decreases during this operation in Lemma~\ref{lem:flattening}. Using it, we will show the first two conjectures (Proposition~\ref{prop:FDC} and Corollary~\ref{prop:FNC}). Then, we will consider another transformation, the lifting, and see in Lemma~\ref{lem:flattening} that the \(m\)-value also decreases with it. It allows us to prove Proposition~\ref{prop:minimality} which links minimal paths with powers of Christoffel words (proving also in the way Conjecture 10.16 from~\cite{Aig15}). Finally, the proof of Proposition~\ref{prop:FSC} solving the fixed sum conjecture ends this paper.

\section{Tools}

\subsection{Words, paths, and their \(m\)-value}

For simplicity of notations, we will use capital letters for denoting matrices, bold letters for denoting words and vectors, and small letters for denoting letters and numbers. For example \(I\) will denote in the following the identity matrix (always of size \(2\times 2\) in this paper).

If \(\bw \in \{a,b\}^\star\) is a word, we will denote by \(\overline{\bw}\) its reversal.

We will mainly consider the words over an alphabet \(\{a,b\}\) of two letters. For each word \(\bw \in \setw\), we associate the corresponding matrices product \(M^{\bw}\) given by the morphism
	\begin{align*}
		M : \{a,b\}^\star & \rightarrow  \rm{SL}_2(\ZZ) \\
			\bw & \mapsto M^\bw
	\end{align*}
	defined by
\[
	A \defeq M^a = \begin{pmatrix} 1 & 1 \\ 1 & 2 \end{pmatrix} \quad \text{and} \quad B \defeq M^b = \begin{pmatrix} 2 & 1 \\ 1 & 1 \end{pmatrix}.
\]
The matrices \(A\) and \(B\) are similar and \(J = J^{-1} = \begin{pmatrix} 0 & 1 \\ 1 & 0\end{pmatrix}\) is their change-of-basis matrix. Moreover, these two matrices are well studied in the context of Frobenius' conjecture since they are generators of the Cohn matrices~\cite{Co55} (it will be especially used in this paper to state Proposition~\ref{prop:cohnMarkov}, but the interested reader can find much more information in~\cite{Aig15}).

We also associate to any \(\bw \in \{a,b\}^\star\) its \(m\)-value (we identify a matrix of order \(1 \times 1\) with its unique entry)
\[
	m(\bw) = \begin{pmatrix}1&0\end{pmatrix} \cdot M^{\bw} \cdot \begin{pmatrix} 0 \\ 1 \end{pmatrix}.
\]
We choose the name \(m\)-value, as it is a generalization of the Markov numbers as we will see in Proposition~\ref{prop:cohnMarkov}.

As usually, we will identify words over \(\{a,b\}\) with their path from \((0,0)\) in the integer lattice \(\NN \times\NN\) (the letter \(a\) corresponds to a horizontal step, and \(b\) to a vertical one).

For example, we associate to the word \(aabab\), the product of matrices 
\[M^{aabab} = A\cdot A \cdot B\cdot A\cdot B = \begin{pmatrix}41 & 29 \\ 65 & 46 \end{pmatrix} \]
and the value 
\[m(aabab) = \begin{pmatrix}1&0\end{pmatrix} \cdot \begin{pmatrix}41 & 29 \\ 65 & 46 \end{pmatrix} \cdot \begin{pmatrix} 0 \\ 1 \end{pmatrix}
 = 29.\]

\begin{figure}[h]
	\includegraphics[scale=.3]{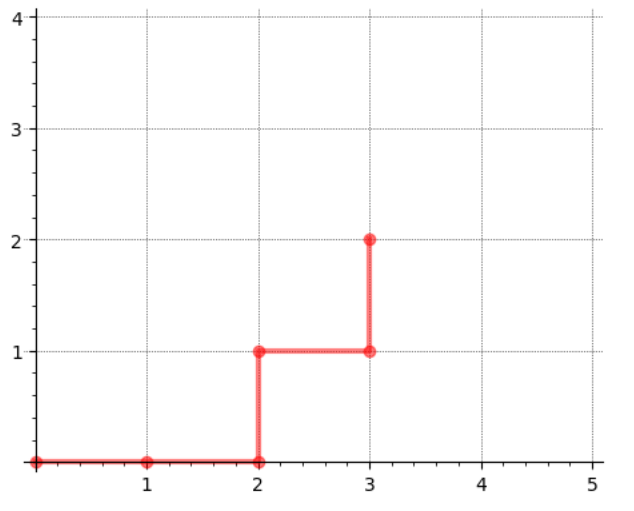}
	\centering
	\caption{Path associated to the word \(aabab\).}
	\label{fig:aabab}
\end{figure}

The corresponding path is showed in Figure~\ref{fig:aabab}.

\subsubsection*{Christoffel words and Markov numbers}

We will consider particularly Christoffel words. 

Let \(u\) and \(v\) be two natural numbers relatively prime. Following~\cite{BLRS09}, the (lower) Christoffel path \(\chris{v/u}\) of slope \(v/u\) is the path (and so the corresponding word by identification) from \((0,0)\) to \((u,v)\) in the integer lattice \(\NN \times \NN\) that satisfies the following two conditions:
\begin{itemize}
	\item the path lies below the line segment that begins at the origin and ends at \((u,v)\),
	\item and the region in the plane enclosed by the path and the line segment contains no other points of \(\NN \times \NN\) besides those of the path and of the line segment.
\end{itemize}

We will call \(c(v,u)\), a generalization of this notion when \(u\) and \(v\) are not relatively prime. The generalized Christoffel path \(c(v,u)\) is the path from \((0,0)\) to \((u,v)\) in the integer lattice \(\NN \times \NN\) that satisfies the following three conditions:
\begin{itemize}
	\item the path lies below the line segment that begins at the origin and ends at \((u,v)\),
	\item all points contained in the line segment are also contained in the path,
	\item and the region in the plane enclosed by the path and the line segment contains no other points of \(\NN \times \NN\) besides those of the path.
\end{itemize}

In particular a generalized Christoffel is just a power of the original Christoffel word of same slope. If \(\gcd(u,v)=\delta\), then \(c(v,u)\) is the word \((\chris{v/u})^{\delta}\).

We stated before that the \(m\)-value is related to the Markov numbers. In fact the Markov numbers are exactly the set of the \(m\)-values of the lower Christoffel words. 
\begin{proposition}\label{prop:cohnMarkov}
	If \(n\) and \(d\) are relatively prime with \(n \leq d\), then \(m(c(n,d))\) corresponds to the Markov number associated to the Farey fraction \(n/d\).
\end{proposition}
This proposition is in fact a restatement of Theorem 7.6 in~\cite{Aig15}. As our objects are defined quite differently, this conversion between the two statements is detailed in Appendix~\ref{app:CohnW}.

The main idea in this work is to associate a Markov number not only to the Farey fractions \(n/d\)
\begin{itemize}
	\item but also to any couple of integers \((d,n)\) (we associate \(m(c(n,d))\)),
	\item and even more to any path \(\bw\) in \(\NN\times\NN\) (we associate \(m(\bw)\)).
\end{itemize} 

In fact, Proposition~\ref{prop:minimality} will show that the numbers \(m(c(n,d))\) are exactly the numbers \(m_{n,d}\) which we defined in the Introduction.

\subsection{\(m\)-values of different words}

\subsubsection{First properties on the \(m\)-values of the words \(c(n,d)\)}

We introduce the following morphism (i.e., satisfying \(E(\bu \bv)=E(\bu)(\bv)\) for all \(\bu,\bv \in \{a,b\}^\star\) )
\[
E : \{a,b\}^\star \rightarrow \{a,b\}^\star \] defined by \(E(a)=b\), \(E(b)=a\) (the name \(E\) stands for `Exchange of letters'). 
We will denote by \(\overline{E}\) the antimorphism we get by composing \(E\) with the reversal transformation \(\overline{E} \defeq E \circ (\bw \mapsto \overline{\bw})\) (one can notice that the two transformations commute \(\overline{E(\bw)} = E(\overline{\bw}) \), so we can equivalently define \(\overline{E}\) by \((\bw \mapsto \overline{\bw}) \circ E\)).

It is well-known that \(E\) is the transformation which relates a lower Christoffel of slope \(\rho\) with the upper Christoffel of slope \(\rho^{-1}\) (see for example Lemma~2.6 in~\cite{BLRS09}, the proof is just the symmetry with respect to the first diagonal). Moreover, it is also known that the reversal of a upper Christoffel word is the lower Christoffel word with the same slope (see for example Proposition 4.2 in~\cite{BLRS09}).  In particular \(\overline{E}(\chris{\rho}) = \chris{\rho^{-1}}\). This is still true for generalized Christoffel words:
\begin{claim}\label{clm:E-paths}
	For any \((d,n) \in \NN\times\NN\), \(c(d,n) = \overline{E}(c(n,d))\).
\end{claim}
\begin{proof}
	Let \(\delta\) be the \(\gcd\) of \(n\) and \(d\). Then, \(c(d,n)  = (\chris{d/n})^{\delta} = \left(\overline{E}(\chris{n/d})\right)^{\delta} = \overline{E}\left((\chris{n/d})^{\delta}\right) = \overline{E}(c(n,d)) \).
\end{proof}

This exchange of letters can be easily restated in terms of matrices
\begin{claim}\label{clm:E-matrices}
	Let \(\bw \in \{a,b\}^\star\), then
	\[
		M^{\overline{E}(\bw)} = \left(J M^\bw  J\right)^T.
	\]
\end{claim}

\begin{proof}
	We recall that \(J\) is involutory and is the change-of-basis matrix between \(A\) and \(B\). Moreover \(A\), \(B\) and \(J\) are symmetric.
	The proof is by induction on the size of \(\bw\).
	
	If \(\bw\) is the empty word then we get the identity matrix on both sides of the equation.
	
	Otherwise, \(\bw = \bu x\) where \(\bu\in \{a,b\}^\star\) and \(x \in \{a,b\}\). By induction hypothesis,
	\begin{align*}
		M^{\overline{E}(\bu x)} & = M^{E(x)}M^{\overline{E}(\bu)} \\
		& =  J M^x J \left(J M^\bu J\right)^T \\
		& = \left(J  M^{\bu x} J \right)^T.
	\end{align*}
\end{proof}

Considering all couples \((d,n)\) instead of those which are relatively prime with \(n<d\), we lose the unicity conjecture:

\begin{lemma}\label{lem:SymFD}
	Let \((d,n)\in \NN^2\), then \(m(c(d,n)) = m(c(n,d))\).
\end{lemma}

\begin{proof}
	More generally for \(\bw \in \{a,b\}^\star\), we have,
	\begin{align*}
		m\left(\overline{E}(\bw)\right) & = \begin{pmatrix} 1 & 0	\end{pmatrix} M^{\overline{E}(\bw)} \begin{pmatrix} 0\\ 1 \end{pmatrix} \\
		& = \left(\begin{pmatrix} 1 & 0	\end{pmatrix} M^{\overline{E}(\bw)} \begin{pmatrix} 0\\ 1 \end{pmatrix}\right)^T \\
		& = \begin{pmatrix} 0 & 1	\end{pmatrix} J M^{\bw} J  \begin{pmatrix} 1\\ 0 \end{pmatrix} \\
		& = 	\begin{pmatrix} 1 & 0	\end{pmatrix} M^{\bw} \begin{pmatrix} 0\\ 1 \end{pmatrix} \\
		& = m(\bw).
	\end{align*}
\end{proof}

In Figure~\ref{fig:firstM}, we computed the first values of \(m(c(n,d))\). We notice that when \(n\) and \(d\) are relatively prime, we get the Markov numbers.

\begin{figure}[h]
	\centering
	{
		\tiny
		\begin{equation*}
		\left[\begin{array}{rrrrrrrrrrr}
			6765 & 10946 & 23763 & 51641 & 112908 & 249755 & 562467 & 1278818 & 2910675 & 6625109 & 15994428 \\
			2584 & 4181 & 9077 & 19760 & 43261 & 96557 & 219472 & 499393 & 1136689 & 2744210 & \color{red} 6625109 \\
			987 & 1597 & 3468 & 7561 & 16725 & 37666 & 85683 & 195025 & 470832 & \color{red} 1136689 & 2910675 \\
			377 & 610 & 1325 & 2897 & 6466 & 14701 & 33461 & 80782 & \color{red} 195025 & \color{red} 499393 & \color{red} 1278818 \\
			144 & 233 & 507 & 1120 & 2523 & 5741 & 13860 & \color{red} 33461 & 85683 & 219472 & 562467 \\
			55 & 89 & 194 & 433 & 985 & 2378 & \color{red} 5741 & \color{red} 14701 & \color{red} 37666 & \color{red} 96557 & 249755 \\
			21 & 34 & 75 & 169 & 408 & \color{red} 985 & 2523 & \color{red} 6466 & 16725 & \color{red} 43261 & 112908 \\
			8 & 13 & 29 & 70 & \color{red} 169 & \color{red} 433 & 1120 & \color{red} 2897 & \color{red} 7561 & 19760 & \color{red} 51641 \\
			3 & 5 & 12 & \color{red} 29 & 75 & \color{red} 194 & 507 & \color{red} 1325 & 3468 & \color{red} 9077 & 23763 \\
			1 & \color{red} 2 & \color{red} 5 & \color{red} 13 & \color{red} 34 & \color{red} 89 & \color{red} 233 & \color{red} 610 & \color{red} 1597 & \color{red} 4181 & \color{red} 10946 \\
			0 & \color{red} 1 & 3 & 8 & 21 & 55 & 144 & 377 & 987 & 2584 & 6765
		\end{array}\right]
		\end{equation*}
	}
	\caption{\(m\)-values of \(c(n,d)\) for \(0 \leq d,n \leq 10\). We highlight in red the Markov numbers.}
	\label{fig:firstM}
\end{figure}

But we conjecture that repetitions in Lemma~\ref{lem:SymFD} are the only repetitions, in particular we can state a stronger conjecture about unicity: 
\begin{conjecture}
	If \(n \leq d\), \(n^\prime \leq d^\prime\) and if \(m(c(n,d)) = m(c(n^\prime,d^\prime))\), then
	we have 
	\[
		n=n^\prime \ \textrm{ and } \ d= d^\prime.
	\]
\end{conjecture}

We tested the conjecture for all \(m(c(n,d))\) with \(d,n \leq 1000\), and we found no collisions. Using Propositions~\ref{prop:FNC} and~\ref{prop:FDC}, it ensures that if a collision occurs, it will be for a number larger than \(m(c(1,1000)) = \Fibo_{2001} > 6 \cdot 10^{417}\) (where \(\Fibo_{2001}\) is the \(2001^{\textrm{th}}\) Fibonacci number, see next section for more information about Fibonacci numbers).

\subsubsection{Fibonacci and Pell numbers}

It is well known that odd-indexed Fibonacci numbers and odd-indexed Pell numbers are part of the Markov numbers. It is noteworthy to mention that even-indexed Fibonacci and Pell numbers also appear among the \(m\)-values of the \(c(n,d)\).

We recall that the Fibonacci numbers are defined by \(\Fibo_0 = 0\), \(\Fibo_1 = 1\) and \(\Fibo_{n+2} = \Fibo_{n+1}+\Fibo_n\) for \(n \geq 2\). Similarly, Pell numbers are defined by \(\Pell_0 = 0\), \(\Pell_1 = 1\) and \(\Pell_{n+2} = 2\Pell_{n+1}+\Pell_n\) for \(n \geq 2\).

We easily show by induction on \(n \geq 1\) that:
\begin{align}\label{eq:Fibo}
	\begin{pmatrix}
		\Fibo_{2n-1} & \Fibo_{2n} \\
		\Fibo_{2n} & \Fibo_{2n+1}
	\end{pmatrix}
	=
	A^n.
\end{align}
This is verified for \(n=1\), let us assume this is true for some \(n\),
\begin{align*}
	A^{n+1} & = 
	\begin{pmatrix}
		\Fibo_{2n-1} & \Fibo_{2n} \\
		\Fibo_{2n} & \Fibo_{2n+1}
	\end{pmatrix}
	\begin{pmatrix}
		1&1 \\ 1&2
	\end{pmatrix} \\
	&= \begin{pmatrix}
		\Fibo_{2n-1}+\Fibo_{2n} & \Fibo_{2n-1}+2\Fibo_{2n} \\
		\Fibo_{2n}+\Fibo_{2n+1} & \Fibo_{2n}+2\Fibo_{2n+1}
	\end{pmatrix}  \\
	&=  
	\begin{pmatrix}
		\Fibo_{2n+1} & \Fibo_{2n+2} \\
		\Fibo_{2n+2} & \Fibo_{2n+3}
	\end{pmatrix}.
\end{align*}

Similarly, we get by induction on \(n \geq 0\) that:
\begin{align}\label{eq:Pell}
	\begin{pmatrix}
		\Pell_{2n+1} - \Pell_{2n} & \Pell_{2n} \\
		2\Pell_{2n}& \Pell_{2n+1}-\Pell_{2n}
	\end{pmatrix}
	=
	(AB)^n.
\end{align}
The result is true for \(n=0\), let assume it is true at a rank \(n\),
\begin{align*}
	(AB)^{n+1} & = 
	\begin{pmatrix}
		\Pell_{2n+1}-\Pell_{2n} & \Pell_{2n} \\
		2\Pell_{2n} & \Pell_{2n+1}-\Pell_{2n}
	\end{pmatrix}
	\begin{pmatrix}
		3&2 \\ 4&3
	\end{pmatrix} \\
	&= \begin{pmatrix}
		3\Pell_{2n+1}+\Pell_{2n} & 2\Pell_{2n+1}+\Pell_{2n} \\
		4\Pell_{2n+1}+2\Pell_{2n} & 3\Pell_{2n+1}+\Pell_{2n}
	\end{pmatrix}  \\
	&=  
	\begin{pmatrix}
		\Pell_{2n+3}-\Pell_{2n+2} & \Pell_{2n+2} \\
		2\Pell_{2n+2} & \Pell_{2n+3}-\Pell_{2n+2}
	\end{pmatrix}.
\end{align*}

It directly implies the following (the cases 2 and 3 were already well-known)
\begin{claim}
	For \(p \in \NN\), we have
	\begin{itemize}
		\item \(m(c(0,p)) = \Fibo_{2p}\),
		\item \(m(c(1,p)) = \Fibo_{2p+1}\),
		\item \(m(c(p,p+1)) = \Pell_{2p+1}\),
		\item and \(m(c(p,p)) = \Pell_{2p}\).
	\end{itemize}
\end{claim}

\begin{proof}
	It results of the fact that \(c(0,p) = a^p\), \(c(1,p)=a^pb\), \(c(p,p+1) = a(ab)^p\) and \(c(p,p) = (ab)^p\).
\end{proof}

\subsubsection{The flip operation}

Let us start by some direct computations. Let \(U = \begin{pmatrix} 0 & -1 \\ 1 & 0\end{pmatrix}\). 
\begin{property}\label{prop:matAB}
	\begin{itemize}
		\item \(AB-BA = 2U\),
		\item \(A U A =  U \),
		\item \(B U B = U \),
		\item \(A U B = \begin{pmatrix} 1 & 0 \\ 3 & 1\end{pmatrix}\) which is nonnegative,
		\item \(B U A = \begin{pmatrix} -1 & -3 \\ 0 & -1\end{pmatrix}\) which is nonpositive.
	\end{itemize}
\end{property}

\begin{figure}[h]
	\centering
	\includegraphics[scale=0.5]{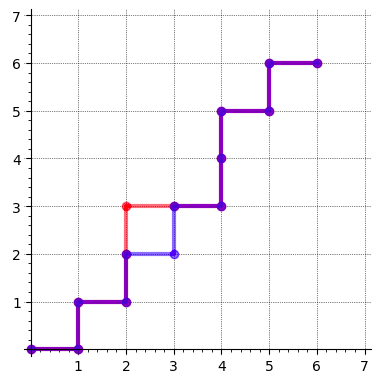}
	\caption{Simple flip between the red word `ababbaabbaba' and the blue word `ababababbaba'} 
	\label{fig:sFlip}
\end{figure}

The purpose of this section is to understand how the \(m\)-value varies when a word \(\overline{\bw_1}ba\bw_2 \in \{a,b\}^\star\) is transformed into a new word \(\overline{\bw_1}ab\bw_2\). Such a transformation (and its reversal) will be called a {\em simple flip}. Figure~\ref{fig:sFlip} shows such a transformation.

We can prove the next result (this lemma is a direct corollary -- where \(\bz\) is the empty word -- of Lemma~\ref{lem:compositeFlip} which will be proved a bit further).

\begin{lemma}\label{lem:simpleFlip}
	Let \( \bw_1, \bw_2 \in \{a,b\}^\star\) and \(\bu\) being the largest common prefix of \(\bw_1\) and \(\bw_2\). One of these cases occurs:
	\begin{enumerate}
		\item If \(\bw_1 = \bu a \bu_1\) and \(\bw_2 = \bu b \bu_2\), then \(m(\overline{\bw_1}ab\bw_2) \geq m(\overline{\bw_1}ba\bw_2)\).
		\item If \(\bw_1 = \bu b \bu_1\) and \(\bw_2 = \bu a \bu_2\), then \(m(\overline{\bw_1}ab\bw_2) < m(\overline{\bw_1}ba\bw_2)\). 
		\item If \(\bw_1 = \bu \) or \(\bw_2 = \bu\), then  \(m(\overline{\bw_1}ab\bw_2) < m(\overline{\bw_1}ba\bw_2)\).
		\item Moreover, (\(\bw_1 = \bu a\) and \(\bw_2 =\bu b\)) if and only if \(m(\overline{\bw_1}ab\bw_2) = m(\overline{\bw_1}ba\bw_2)\). 
	\end{enumerate} 
\end{lemma}

The idea now is to compose these simple flips to transform a word \(\bw_1\) into a new word \(\bw_2\). If the \(m\)-value varies always in the same direction, we can conclude on the final variation. For example, if we want to transform the word \(a^2b^2\) into the word \(b^2a^2\), then
\begin{align*}
	m(aabb) = 12 & = m(abab)  = 12 && \text{by point 4} \\  
	& < m(baab) = 18  && \text{by point 3} \\
	& < m(baba) = 24 && \text{by point 3} \\
	& < m(bbaa) = 30 && \text{by point 2}.
\end{align*}
Consequently \(m(a^2b^2) < m(b^2a^2)\).

However this approach does not always succeed. Indeed, we will be interested in words of the form \(a\bp ab\bp ab \bp b\) where \(\bp \in \{a,b\}^\star\) is a palindrome. Can we conclude how the \(m\)-value varies by flipping the two `ab' to get \(a \bp ba \bp ba \bp b\)? Whatever the order of flips we choose, the \(m\)-value will vary differently at each step. There is a good reason for that: in fact we have 
\[
m(a\bp ab\bp ab \bp b) = m(a \bp ba \bp ba \bp b).
\]

To handle these words, the idea is to flip the two `ab' in the same time. In fact we will directly flip the factor \(ab \bp ab\) to get \(ba \bp ba\). More generally we will be able to flip the factor \(a \bz b\) into a new factor \(b \overline{\bz} a\) (with \(\bz \in \{a,b\}^\star\)), we will call this operation a {\em composite flip}.

The next lemma describes the difference in the matrices associated with these two factors.

\begin{lemma}\label{lem:diff} 
	If \(\bz\in \{a,b\}^\star\), then there exists \(t \in \NN\setminus\{0\}\) such that \(AM^{\bz} B - B M^{\overline{\bz}} A = t U\).
\end{lemma}

\begin{proof}
	We have \((A M^{\bz} B)^T = B M^{\overline{\bz}} A\) so the difference is antisymmetric. Moreover, let us choose \( \begin{pmatrix} \alpha & \beta \\ \gamma & \delta \end{pmatrix} \defeq M^\bz \in \rm{SL}_2(\ZZ)\).
	\begin{align*}
		\begin{pmatrix} 1 & 0 \end{pmatrix} \left(A M^{\bz} B - B M^{\overline{\bz}} A\right)\begin{pmatrix} 0 \\ 1\end{pmatrix} 
		& =  \begin{pmatrix} 1 & 0\end{pmatrix} A M^{\bz} B \begin{pmatrix} 0 \\ 1\end{pmatrix} - \left(\begin{pmatrix} 0 & 1\end{pmatrix}A M^{\bz} B \begin{pmatrix} 1 \\ 0\end{pmatrix} \right)^T \\
		& = \begin{pmatrix} 1 & 1\end{pmatrix} M^{\bz} \begin{pmatrix} 1 \\ 1\end{pmatrix} - \begin{pmatrix} 1 & 2\end{pmatrix}M^{\bz} \begin{pmatrix} 2 \\ 1\end{pmatrix}  \\
		& = -\alpha-\delta-3\gamma \\
		& \leq 0
	\end{align*}
	and is non zero since \(M^{\bz}\) invertible
\end{proof}

\begin{figure}[h]
	\centering
	\includegraphics[scale=0.5]{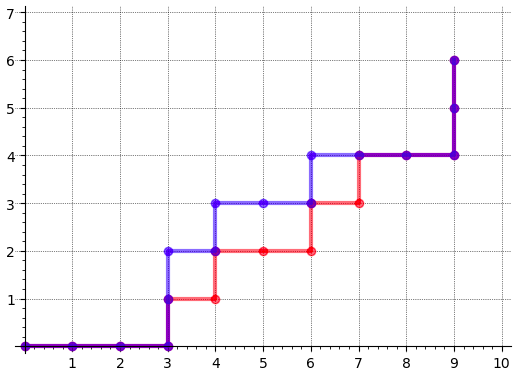}
	\caption{Composite flip between the red word `aaab {\em a baaba b} aabb' and the blue word `aaab {\em b abaab a} aabb'} 
	\label{fig:cFlip}
\end{figure}

We can prove now a generalization of Lemma~\ref{lem:simpleFlip}:
\begin{lemma}\label{lem:compositeFlip}
	Let \( \bz, \bw_1, \bw_2 \in \{a,b\}^\star \) and \(\bu\) being the largest common prefix of \(\bw_1\) and \(\bw_2\). One of these first three cases occur:
	\begin{enumerate}
		\item If \(\bw_1 = \bu a \bu_1\) and \(\bw_2 = \bu b \bu_2\), then \(m(\overline{\bw_1}a\bz b\bw_2) \geq m(\overline{\bw_1}b\overline{\bz}a\bw_2)\).
		\item If \(\bw_1 = \bu b \bu_1\) and \(\bw_2 = \bu a \bu_2\), then \(m(\overline{\bw_1}a\bz b\bw_2) < m(\overline{\bw_1}b\overline{\bz}a\bw_2)\). 
		\item If \(\bw_1 = \bu \) or \(\bw_2 = \bu\), then  \(m(\overline{\bw_1}a\bz b\bw_2) < m(\overline{\bw_1}b\overline{\bz}a\bw_2)\).
		\item Moreover, (\(\bw_1 = \bu a\) and \(\bw_2 =\bu b\)) if and only if \(m(\overline{\bw_1}a\bz b\bw_2) = m(\overline{\bw_1}b\overline{\bz}a\bw_2)\). 
	\end{enumerate} 
\end{lemma}

\begin{proof}
	First let us show the implications of the first three cases (and the necessary condition of the last point).
	
	In the first case, using Property~\ref{prop:matAB} and Lemma~\ref{lem:diff},
	\begin{align*}
		m(\overline{\bw_1}a\bz b\bw_2) - m(\overline{\bw_1}b\overline{\bz}a\bw_2)  & = \begin{pmatrix} 1 & 0\end{pmatrix} M^{\overline{\bu_1}a\overline{\bu}} (AM^{\bz}B-BM^{\overline{\bz}}A) M^{\bu b\bu_2}	\begin{pmatrix} 0 \\ 1\end{pmatrix}  \\
		& = t \begin{pmatrix} 1 & 0\end{pmatrix} M^{\overline{\bu_1}} A M^{\overline{\bu}} U M^{\bu} B M^{\bu_2}	\begin{pmatrix} 0 \\ 1\end{pmatrix}  \\
		& = t \begin{pmatrix} 1 & 0\end{pmatrix} M^{\overline{\bu_1}} A U B M^{\bu_2}	\begin{pmatrix} 0 \\ 1\end{pmatrix}  \\
		& = t \begin{pmatrix} 1 & 0\end{pmatrix} M^{\overline{\bu_1}} \begin{pmatrix} 1 & 0\\ 3 &1\end{pmatrix} M^{\bu_2}	\begin{pmatrix} 0 \\ 1\end{pmatrix}  \\
		& \geq 0
	\end{align*}
	and the equality occurs if and only if \(\bu_1\) and \(\bu_2\) are the empty words.
	
	In the second case,
	\begin{align*}
		m(\overline{\bw_1}a\bz b\bw_2) - m(\overline{\bw_1}b\overline{\bz}a\bw_2)  & =  \begin{pmatrix} 1 & 0\end{pmatrix} M^{\overline{\bu_1}b\overline{\bu}} (AM^\bz B-BM^{\overline{\bz}}A) M^{\bu a\bu_2}	\begin{pmatrix} 0 \\ 1\end{pmatrix}  \\
		& = t \begin{pmatrix} 1 & 0\end{pmatrix} M^{\overline{\bu_1}} BUA M^{\bu_2}	\begin{pmatrix} 0 \\ 1\end{pmatrix}  \\
		& = t \begin{pmatrix} 1 & 0\end{pmatrix} M^{\overline{\bu_1}} \begin{pmatrix} -1 & -3\\ 0 &-1\end{pmatrix} M^{\bu_2}	\begin{pmatrix} 0 \\ 1\end{pmatrix}  \\
		& < 0.
	\end{align*}
	Let us assume now that \(\bu =\bw_1\), then, by defining \(\bw_2 = \bu\bu_2\),
	\begin{align*}
		m(\overline{\bw_1}a\bz b\bw_2) - m(\overline{\bw_1}b\overline{\bz}a\bw_2)  & = t \begin{pmatrix} 1 & 0\end{pmatrix}  U M^{\bu_2}	\begin{pmatrix} 0 \\ 1\end{pmatrix}  \\
		& = t \begin{pmatrix} 0 & -1\end{pmatrix}  M^{\bu_2}	\begin{pmatrix} 0 \\ 1\end{pmatrix}  \\
		& < 0.
	\end{align*}
	The case \(\bu = \bw_2\) is similar. By defining \(\bw_1 = \bu\bu_1\) we get
	\begin{align*}
		m(\overline{\bw_1}a\bz b\bw_2) - m(\overline{\bw_1}b\overline{bz}a\bw_2)  & = t \begin{pmatrix} 1 & 0\end{pmatrix}  M^{\overline{\bu_1}} U	\begin{pmatrix} 0 \\ 1\end{pmatrix}  \\
		& = t \begin{pmatrix} 1 & 0\end{pmatrix}  M^{\overline{\bu_1}}	\begin{pmatrix} -1 \\ 0\end{pmatrix}  \\
		& < 0.
	\end{align*}
	
	Finally, one can notice that if none of these cases occur, then \(\bu\) is a strict prefix of \(\bw_1=\bu\bu_1\) and of \(\bw_2=\bu\bu_2\). Furthermore, \(\bu_1\) and \(\bu_2\) have to start by the same letter contradicting the fact that \(\bu\) is the largest common prefix. That shows we met all cases which could occur. Furthermore, it also implies the sufficient condition of the fourth point.
\end{proof}

In particular if \(\bp\) is a palindrome, applying the last point with \(\bw_1 = \bp a\) and \(\bw_2 = \bp b\), we directly get
\[
m(a\bp ab \bp ab \bp b) = m(a \bp ba \bp ba \bp b)
\]
as we claimed before.

In fact the last inequality can also be deduced from the next lemma 
\begin{lemma}\label{lem:ab}
	Let \(\bw \in \setw\). Then \(m(a\bw b) = m(a \overline{\bw} b)\).
\end{lemma}
\begin{proof}
	The matrix \(M^{\bw} - M^{\overline{\bw}}\) is antisymmetric, so of the form \(\ell U\) where \(\ell \in \ZZ\). Hence,
	\[
		m(a\bw b) - m(a \overline{\bw}b) = \ell \begin{pmatrix}
			1 & 0
		\end{pmatrix}
		AUB \begin{pmatrix}
			0 \\ 1
		\end{pmatrix} =0.
	\]
\end{proof}

\subsection{Below sets and hulls}

\begin{figure}
	\centering
	\begin{subfigure}{.4\textwidth}
		\centering
		\includegraphics[scale=0.4]{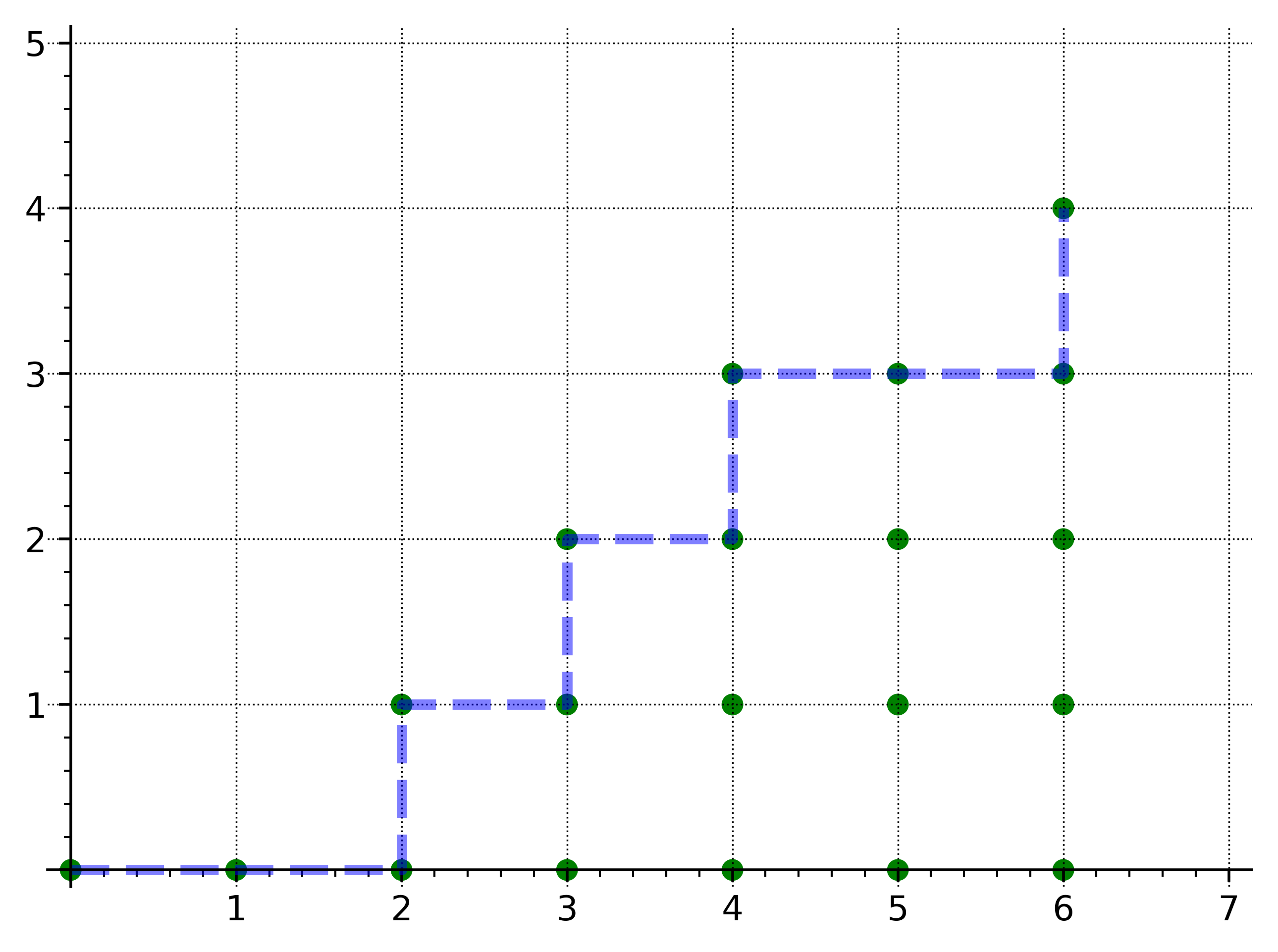}
		\caption{The green points form the set \(\bel(aabababaab)\). \\
			 Similarly, the dotted blue line is the hull of the \(6\)-packed set formed by the green points.} 
		 \label{fig:bS1}
	\end{subfigure}%
	\begin{subfigure}{.15\textwidth}
		\ 
	\end{subfigure}
	\begin{subfigure}{.4\textwidth}
		\centering
		\includegraphics[scale=0.4]{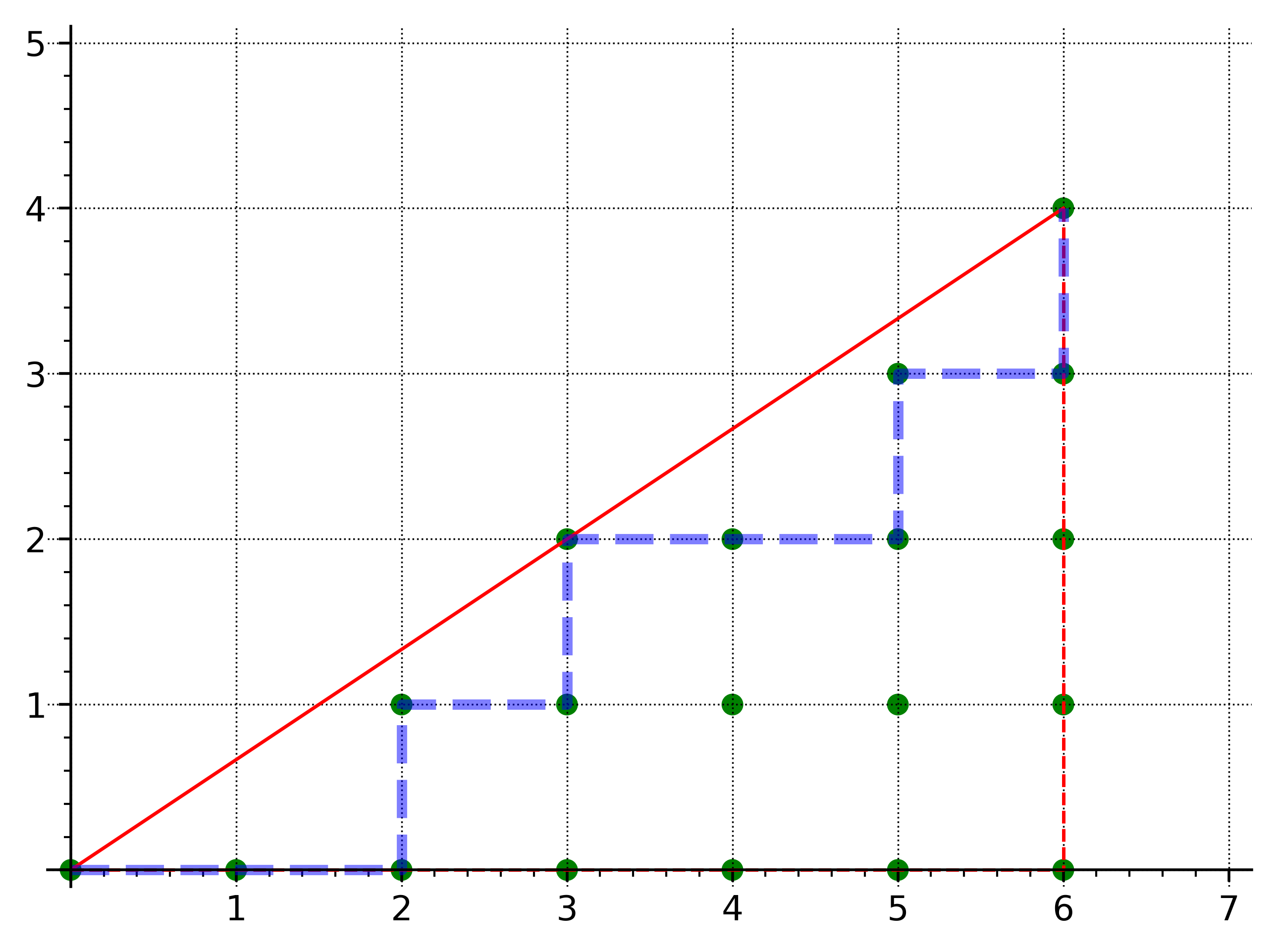}
		\caption{The green points form the set \(T_{4,6} = \bel(c(4,6))\). \\
		The blue path is \(c(4,6)\) which is also \(\hull(T_{4,6})\).} 
		\label{fig:bS2}
	\end{subfigure}
	\caption{Examples of \(\bel\) sets}
	\label{fig:belSets}
\end{figure}

The notation \([d]\) will correspond to the integer interval \(\{n \in \NN \mid n \leq d\}\). A subset of \(\NN^2\) is said to be of {\em width} \(d\) if it is a subset of \([d]\times \NN\).

A subset \(S \subseteq \NN^2\) is called {\em \(d\)-packed} if it is a finite set of width \(d\) satisfying the three following conditions 
\begin{itemize}
	\item \((0,0) \in S\),
	\item \(\forall (u,v) \in[d-1]\times\NN, (u,v) \in S\implies (u+1,v)\in S \),
	\item and  \(\forall (u,v) \in[d]\times(\NN\setminus\{0\}), ((u,v) \in S)\implies (u,v-1)\in S \).
\end{itemize}

A set will be called {\em packed}, if it is \(d\)-packed for some \(d \in \NN\). We notice that given a packed set \(S\), there exists a unique \(d\) such that it is \(d\)-packed: \(
	d =\max_{(u,v) \in S}(u)
\).

We show that finite words on a two letters alphabet are in one-to-one correspondence with packed sets.

Let \(\bw\) be a word with \(d\) `a' and \(n\) `b'. We saw that \(\bw\) can be identified as a path from \((0,0)\) to \((d,n)\) in the \(\NN^2\) lattice. We define \(\bel(\bw)\) as the set of points below the word \(\bw\). More formally,
\[
	\bel(\bw) = \{ (u,v) \in [d]\times \NN \mid \exists p \geq v \text{ such that }\bw\text{ goes through }(u,p)\}.
\]

For example, in Figure~\ref{fig:bS1}, the green points are the elements of the set \(\bel(aabababaab)\).

The next claim is immediate.
\begin{claim}
	Let \(\bw \in \{a,b\}^\star\). Then \(\bel(\bw)\) is a packed set.
\end{claim}

In fact this transformation is invertible, it is possible to retrieve the word \(\bw\) given the set \(\bel(\bw)\).

Indeed, given a packed set \(S\), let \(\hull(S)\) be the upper hull of \(S\)
\[
	\hull(S) = \{(x,y) \in S \mid (x-1,y+1) \notin S\}.
\]

Back to the Figure~\ref{fig:bS1}, we can also see the dotted blue line as the hull of the set of the green points.

Consequently
\begin{lemma}
	The mapping \(\bel : \{a,b\}^\star \rightarrow \{\text{packed sets}\}\) is one-to-one of inverse \(\bord\).
\end{lemma}

Let us apply this construction for \(c(n,d)\) words. Given \((d,n) \in \NN^2\), let \(T_{n,d}\) be the triangle delimited by \((0,0)\), \((d,0)\) and \((d,n)\).
\[
	T_{n,d} = \{(u,v) \in \NN^2 \mid u \leq d \text{ and } dv \leq un\}.
\]

We have
\begin{claim}
	Let \((d,n) \in \NN^2\), then \(\bel(c(n,d)) = T_{n,d}\).
\end{claim}

In Figure~\ref{fig:bS2}, the hull is \(c(4,6)\) and the set of below points is \(T_{4,6}\).

\section{Fixed parameter conjectures}

The fixed denominator and numerator conjectures can be proved using only simple flips. In fact they are direct consequences of Lemma~\ref{lem:flattening} which states that the \(m\)-value decreases when we flatten a word.

\subsection{Flattenings}

Let \(\bw \in \{a,b\}^\star\) be a path in \(\NN^2\) from  \((0,0)\) to \((d,n)\) with \(d \geq 1\). To each point \((x,y) \in [d]\times \NN\), we associate its {\em algebraic vertical distance} to the line linking \((0,0)\) and \((d,n)\):
\begin{align*}
	\delta_\bw : [d]\times \NN & \rightarrow \RR  \\
	(x,y) & \mapsto \frac{dy-nx}{d}
\end{align*}
(we can extend this definition for the words \(\bw \in \{b\}^\star\) by \(\delta_\bw : [0]\times\NN \rightarrow \RR, (0,y) \mapsto y\)).

The name `algebraic vertical distance' comes from the fact that the point \((x, y-\delta_\bw(x,y))=(x,(n/d)x)\) belongs to the line going through \((0,0)\) and \((d,n)\). By construction, \(\delta_\bw((0,0)) = \delta_\bw((d,n))=0\) (they are on the line). Moreover, if \(\delta_\bw(x,y) <0\), it means that \((x,y)\) lies below the line and if \(\delta_\bw > 0\), then the point is above.

If \(d,n \geq 1\), the flattening operation consists in removing from \(\bel(\bw)\) all points which stand strictly above the line.

We define
\begin{align*}
	\flat(\bw) & \defeq \hull\left(\bel(\bw) \cap \delta_\bw^{-1}( \left]-\infty,0\right] ) \right) \\
	& = \hull\left(\bel(\bw) \cap T_{n,d} \right).
\end{align*}
So we already have
\[
	\bel({\flat(\bw)} ) = \bel(\bw) \cap \delta_\bw^{-1}( \left]-\infty,0\right] ) \subseteq \bel(\bw).
\]

Such an example of flattening is presented in Figure~\ref{fig:flattening}.

\begin{figure}
	\centering
	\begin{subfigure}{.4\textwidth}
		\centering
		\includegraphics[scale=0.4]{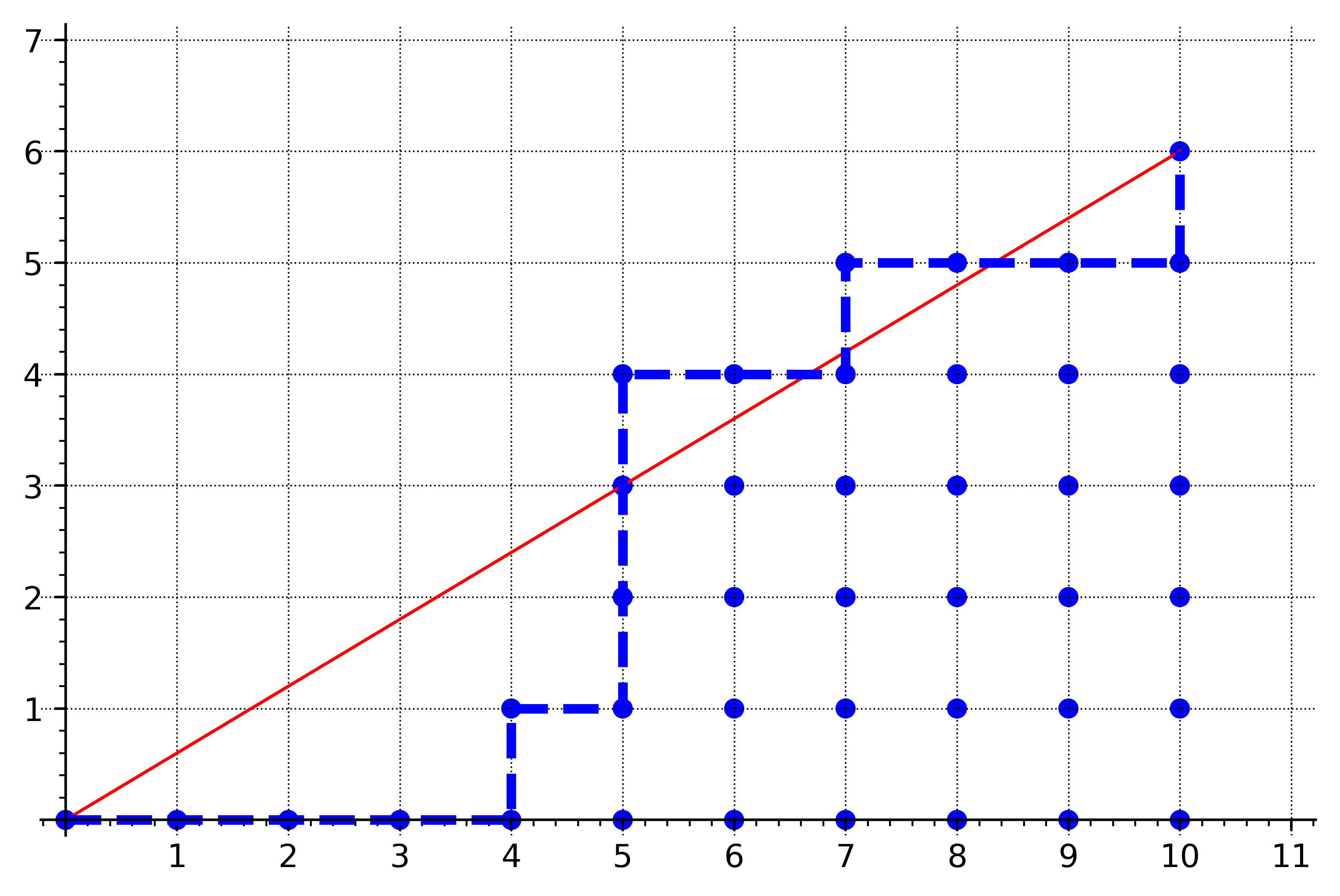}
	\end{subfigure}%
	\begin{subfigure}{.15\textwidth}
		\ 
	\end{subfigure}
	\begin{subfigure}{.4\textwidth}
		\centering
		\includegraphics[scale=0.4]{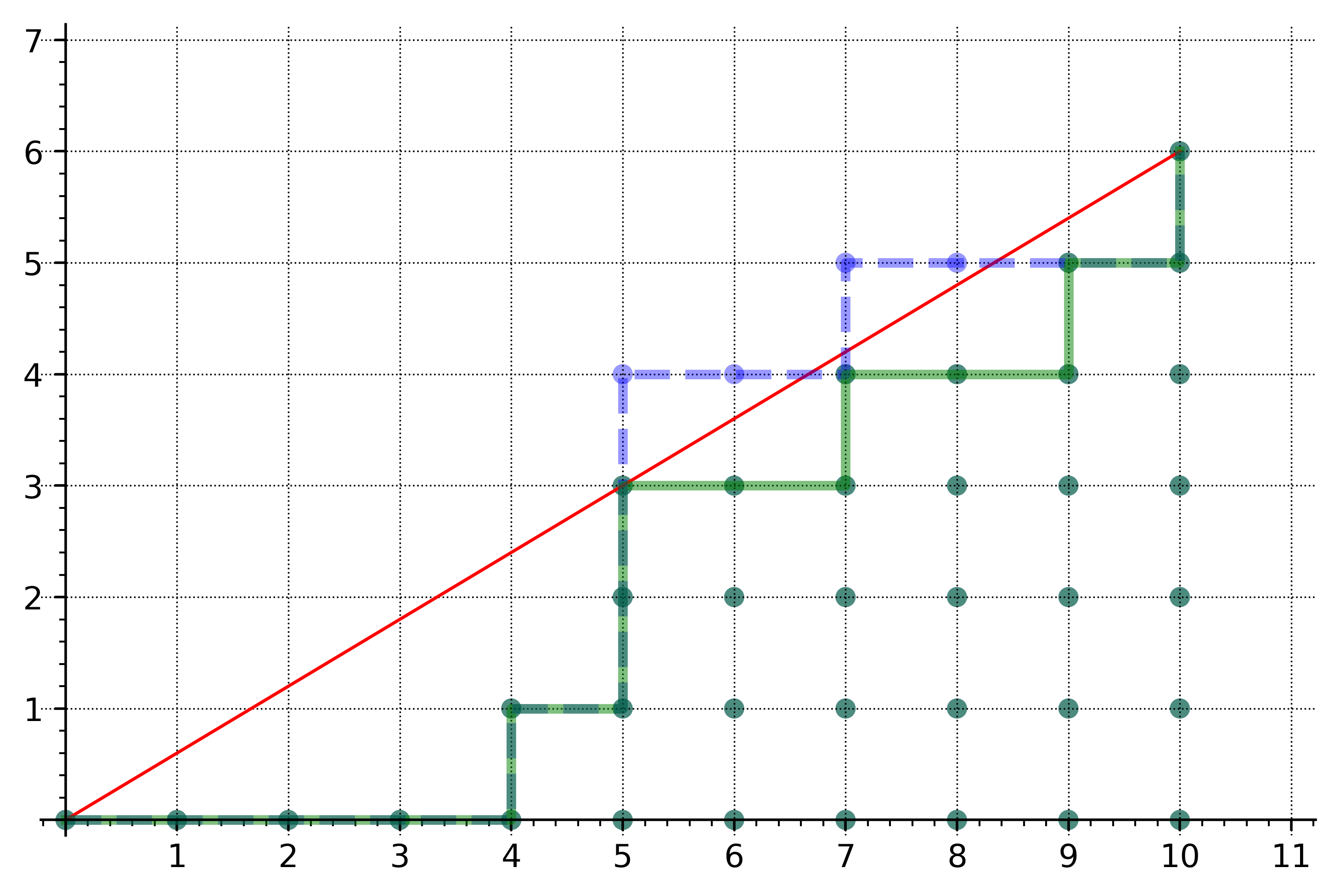}
	\end{subfigure}
	\caption{The green path is the flattening of the blue path.}
	\label{fig:flattening}
\end{figure}

We show now that flattening a word decreases its \(m\)-value. The idea of the proof is just to remove, one by one, each vertex above the line via simple flips. Notice that Lemma~\ref{lem:simpleFlip} ensures that if we are not in the first case of the lemma, a flip from \(\bu ba \bv\) to \(\bu ab \bv\) always decreases the \(m\)-value. So, the only point of the next proof is to find a good order of the points to flip such that the case 1 never occurs.

\begin{lemma}\label{lem:flattening}
	Let \(\bw \in \{a,b\}^\star\). Then\[
		m(\flat(\bw)) \leq  m(\bw),
	\]
	and the equality stands only if \(\flat(\bw)  = \bw\).
\end{lemma}

\begin{proof}
	Let \(\bw \in \{a,b\}^\star\) with \(d\) `a' and \(n\) `b'.
	
	We show this lemma by induction on the size of \(\bel(\bw)\setminus\bel(\flat(\bw))\).
	
	If this set is empty, it implies that \(\bel(\flat(\bw)) = \bel(\bw)\), i.e., \(\flat(\bw) = \bw\) and so \(m(\flat(\bw)) =  m(\bw)\).
	
	Otherwise, let \((x,y)\) be the integer point in \(\bel(\bw)\) such that the couple \((\delta_\bw(x,y),x)\) is maximal (via the lexicographical order of \(\RR^2\)). As \(\bel(\bw) \neq \bel(\flat(\bw))\), it implies there exist points in \(\bel(\bw)\) with positive \(\delta_\bw\). So, \(\delta_\bw(x,y) > 0\). Using again the maximality of \(\delta_\bw(x,y)\), \((x,y)\) lies on the hull of \(\bel(\bw)\), i.e., on the path \(\bw\). Let \(\bu\) be the prefix of \(\bw\) such that the path \(\bu\) ends at \((x,y)\) and let \(\bv\) such that \(\bw = \bu \bv\). As \(\delta_\bw(x,y)>0\), we know that \(\bu\) and \(\bv\) are non empty. If \(\bu\) ends by a `a', it would imply that \((x-1,y)\) lies on \(\bw \). However, \(\delta_\bw(x-1,y) = (dy-nx+n)/d > \delta_\bw(x,y)\) which would contradict the maximality. So \(\bu\) ends with a `b'. Similarly, as \(\delta_\bw(x,y+1) > \delta_\bw(x,y)\), \(v\) starts with the letter `a'. Consequently, \(\bw = (\bu^\prime b)(a \bv^\prime)\). We would like to flip this factor `ba'. To ensure the \(m\)-value decreases during the flip, we just need to show that we are not in the first case of Lemma~\ref{lem:simpleFlip}. Let \(\bz\) be the largest common prefix of \(\overline{\bu^\prime}\) and \(\bv^\prime\). 
	
	Assume that \(\overline{\bu^\prime} = \bz a \overline{\bu^{\prime\prime}}\) and \(\bv^\prime = \bz b \bv^{\prime\prime}\). 
	Let \((x_z,y_z)\) be the endpoint of the path \(\bz\), \((x_1,y_1)\) be the one of \(\bu^{\prime\prime}\) and \((x_2,y_2)\) be the one of \(u^\prime ba \bz b\) (the starting point is always implicitly the origin \((0,0)\)). 
	We have,
	\[
		(x_1,y_1)+(1,0)+(x_z,y_z) + (0,1)= (x,y)
	\]
	and
	\[
		(x,y)+(1,0)+(x_z,y_z)+(0,1) = (x_2,y_2).
	\]
	Hence,
	\begin{align*}
		\delta_\bw(x_1,y_1)+\delta_\bw(x_2,y_2) &=  \frac{d(y_1+y_2)-n(x_1+x_2)}{d} 
		& =  \frac{2dy-2nx}{d} \\
		& = 2 \delta_\bw(x,y).
	\end{align*}
	By maximality of \(\delta_\bw(x,y)\), it would imply that \(\delta_\bw(x_1,y_1) = \delta_\bw(x_2,y_2)=\delta_\bw(x,y)\). But, as \(x_2>x\), it contradicts the maximality of the couple \((\delta_\bw(x,y),x)\).
	
	So we are not in case 1 of Lemma~\ref{lem:simpleFlip}. Hence, as \(\bw = \bu^\prime ba \bv^\prime\),
	\begin{equation}\label{eq:flattening1}
		m(\bw) > m(\bu^\prime ab \bv^\prime).
	\end{equation}
	
	By construction, \(\bel(\bu^\prime ab \bv^\prime) = \bel(\bw)\setminus\{(x,y)\}\) and \(\delta_\bw(x,y) >0\). So,
	\(\flat(\bu^\prime ab \bv^\prime) = \flat(\bw)\). Finally, by induction hypothesis, 
	\begin{equation}\label{eq:flattening2}
		m(\bu^\prime ab \bv^\prime) \geq m(\flat(\bw)).
	\end{equation}
	Combining Equations (\ref{eq:flattening1}) and (\ref{eq:flattening2}) concludes the induction and so proves the lemma.
\end{proof}

\subsection{Application to the fixed denominator and numerator conjectures}

The fixed denominator conjecture is an easy consequence of Lemma~\ref{lem:flattening}:
\begin{proposition}\label{prop:FDC}
	Let \((d,n) \in \NN^2\). We have \(m(c(n,d)) < m(c(n+1,d))\).
\end{proposition}

\begin{proof}
	If \(n=0\), then \(m(c(0,d)) = \Fibo_{2d} < \Fibo_{2d+1} = m(c(1,d))\). Moreover, if  \(d=0\), then \(m(c(n,0)) = \Fibo_{2n} < \Fibo_{2n+2} = m(c(n+1,0))\). So we can assume in the following that \(d,n \geq 1\).
	
	We know that \(c(n+1,d)\) ends with a `b'. So let us define \(\bw b \defeq c(n+1,d)\). We have
	\[m(\bw b) - m(\bw) = \begin{pmatrix}
		1 & 0
	\end{pmatrix} M^\bw (B-I) \begin{pmatrix}
	0 \\ 1
\end{pmatrix} > 0.
	\]
	
	Now, \(\bel(\bw) = \bel(c(n+1,d))\setminus\{(d,n+1)\} = T_{n+1,d}\setminus\{(d,n+1)\}\). Hence, as \(T_{n,d} \subseteq T_{n+1,d}\setminus\{(d,n+1)\}\),
	\[
	\flat(\bw) = \hull\left((T_{n+1,d}\setminus\{(d,n+1)\}) \cap T_{n,d} \right) = \hull(T_{n,d}) = c(n,d).
	\]
	
	Then by Lemma~\ref{lem:flattening}, 
	\[
		m(c(n+1,d)) > m(\bw) \geq m(c(n,d)).
	\]
\end{proof}

As we proved Proposition~\ref{prop:FDC} for all couples \((d,n) \in \NN^2\), it directly implies the fixed numerator conjecture which was already solved in~\cite{RS20}:
\begin{corollary}\label{prop:FNC}
	Let \((d,n) \in \NN^2\). We have \(m(c(n,d)) < m(c(n,d+1))\).
\end{corollary}

\begin{proof}
	This is directly implied by Lemma~\ref{lem:SymFD} and Proposition~\ref{prop:FDC}: \(m(c(n,d))  = m(c(d,n)) < m(c(d+1,n)) =  m(c(n,d+1))\).
\end{proof}

\subsection{Liftings and minimality for Christoffel like words}

We saw before how to flatten a path until reaching the corresponding Christoffel. We handle here the opposite direction: we want to raise a path until reaching the Christoffel.

We recall the definition of the algebraic vertical distance: 
\begin{align*}
	\delta_\bw : [d]\times \NN & \rightarrow \RR  \\
	(x,y) & \mapsto \frac{dy-nx}{d}.
\end{align*}

The {\em lifting} operation consists in adding to \(\bel(\bw)\) all points which stand strictly below the line. We define
\begin{align*}
	\lift(\bw) & \defeq \hull\left(\bel(\bw) \cup \delta_\bw^{-1}( \left]-\infty,0\right[ ) \right).
\end{align*}
So we have
\[
\bel({\lift(\bw)} ) = \bel(\bw) \cup \delta_\bw^{-1}( \left]-\infty,0\right[ ).
\]

One can wonder why \(0\) is withdrawn of the interval just above. In fact, to reconstruct Christoffel words, we should add the \(0\). However, we give two reasons for our choice:
\begin{itemize}
	\item it would need more complex operations (for example, simple flips would not be sufficient anymore), so it makes sense to treat this case later alone,
	\item furthermore, looking at Proposition~\ref{prop:minimality}, the extremal words are not only the Christoffel but also close variations which are preserved with this definition.
\end{itemize} 

Such an example of lifting is presented in Figure~\ref{fig:lifting}.

\begin{figure}
	\centering
	\begin{subfigure}{.4\textwidth}
		\centering
		\includegraphics[scale=0.4]{FlatteningBefore}
	\end{subfigure}%
	\begin{subfigure}{.15\textwidth}
		\ 
	\end{subfigure}
	\begin{subfigure}{.4\textwidth}
		\centering
		\includegraphics[scale=0.4]{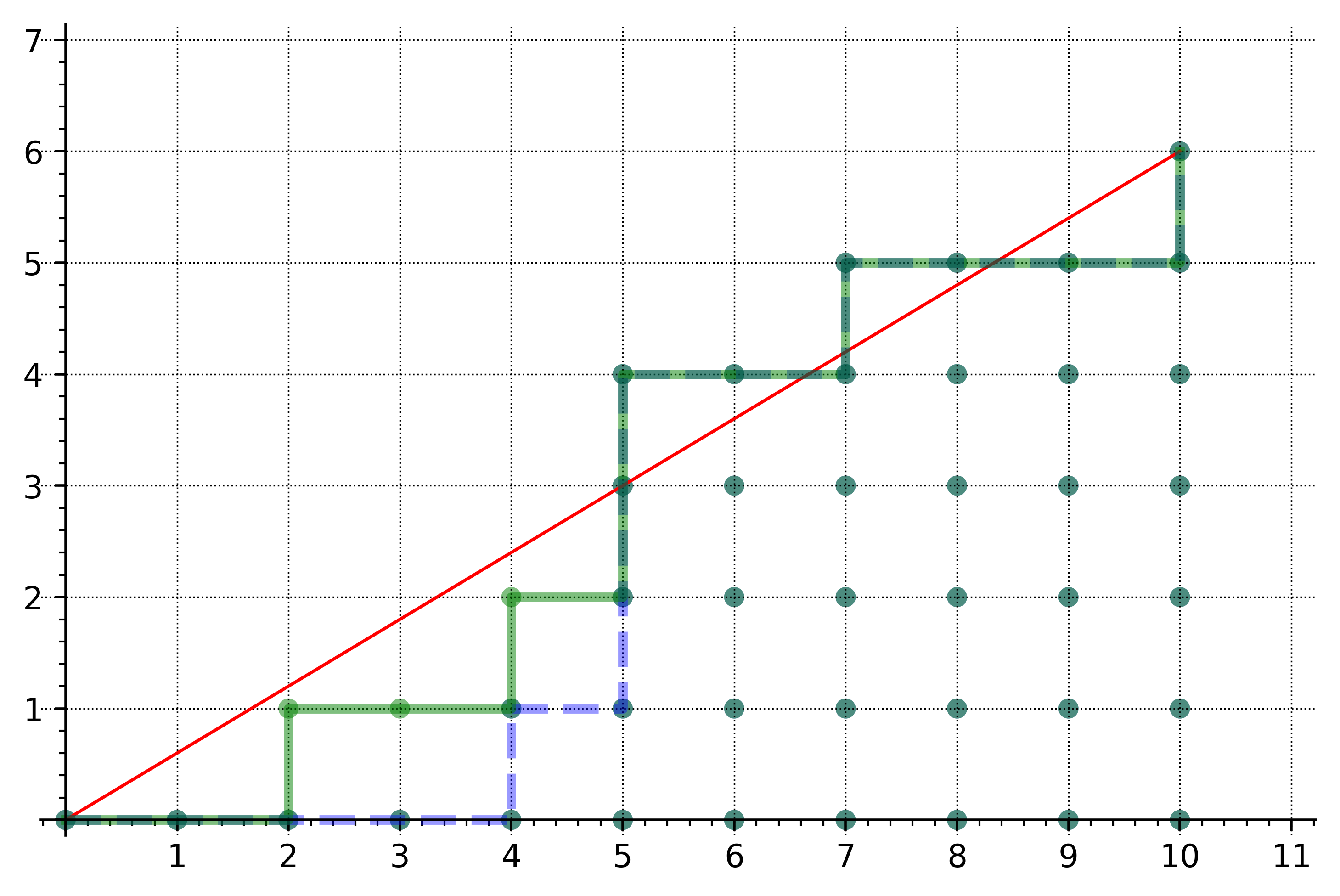}
	\end{subfigure}
	\caption{The green path is the lifting of the blue path.}
	\label{fig:lifting}
\end{figure}

Similarly to the case of the flattenings, the \(m\)-value decreases during a lifting. Also, in the same way, simple flips are enough to get the next Lemma. The proof is quite similar to the one of Lemma~\ref{lem:flattening}, but now as we flip from a word \(\overline{\bw_1}ab \bw_2\) towards the word \(\overline{\bw_1} ba \bw_2\), we want to ensure that we are in the first case of Lemma~\ref{lem:simpleFlip}. 

\begin{lemma}\label{lem:lifting}
	Let \(\bw \in \{a,b\}^\star\). Then\[
	m(\lift(\bw)) \leq  m(\bw).
	\]
	
	Moreover, the equality stands only if \(\lift{(\bw)} = \bw\).
\end{lemma}

\begin{proof}
	Let \(\bw \in \{a,b\}^\star\) with \(d\) `a' and \(n\) `b'.
	
	We show this lemma by induction on the size of \(\bel(\lift(\bw))\setminus\bel(\bw)\).
	
	If this set is empty, it implies that \(\bel(\lift(\bw)) = \bel(\bw)\), i.e., \(\lift(\bw) = \bw\) and so \(m(\lift(\bw)) =  m(\bw)\).
	
	Otherwise, since \(\bel(\bw) \subsetneq \bel(\lift(\bw)) \), we can extract from \(\bel(\lift(\bw))\setminus\bel(\bw)\) a point \((x^\prime,y^\prime)\) such that \(\delta_\bw(x^\prime,y^\prime)\) is minimal. So, \((x^\prime+1,y^\prime-1)\) belongs to the path \(\bw\) and verifies 
	\begin{align*}
		\delta_\bw(x^\prime+1,y^\prime-1) &= \frac{d(y^\prime -1)-n(x^\prime+1)}{d} \\
		&= \delta_\bw(x^\prime,y^\prime) -1 -n/d \\
		& < -1-n/d.
	\end{align*}
	It implies that \((x,y)\), the point on \(\bw\) such that the couple \((\delta_\bw(x,y),x)\) is minimal, satisfies \(\delta_\bw(x,y)<-1-n/d\). In particular it is not an extremity of \(\bw\). Let \(\bu\) be the prefix of \(\bw\) such that the path \(\bu\) ends at \((x,y)\) and let \(\bv\) such that \(\bw = \bu \bv\). Again by minimality, we know that \(\bu\) ends with an `a' and \(\bv\) starts with a `b'. So \(\bw = \bu^\prime ab \bv^\prime\). We would like to flip the factor `ab'. To ensure the \(m\)-value decreases during the flip, we need to show that we are in the first case of Lemma~\ref{lem:simpleFlip}. Let \(\bz\) be the largest common prefix of \(\overline{\bu^\prime}\) and \(\bv^\prime\). So \(\bw = \bu_1 \overline{\bz} ab \bz \bv_1\).
	
	Let \((x_1,y_1)\) be the vertex at the end of \(\bu_1\) and let \((x_2,y_2)\) be the one at the end of \(\bu_1 \overline{\bz}ab \bz\).
	We have 
	\begin{align*}
		\delta_\bw(x_1,y_1)+\delta_\bw(x_2,y_2) &=  \frac{d(y_1+y_2)-n(x_1+x_2)}{d} 
		& =  \frac{d(2y+1)-n(2x-1)}{d} \\
		& = 2 \delta_\bw(x,y) +(1+n/d) \\
		&< \delta_\bw(x,y).
	\end{align*}
	Assume that \(\bu_1\) or \(\bv_1\) is the empty word. In these cases, one of the two terms \(\delta_\bw(x_1,y_1)\) and \(\delta_\bw(x_2,y_2) \) is zero. Consequently the other one would be smaller than \(\delta_\bw(x,y)\) which would contradict the minimality. So we are not in the third case of Lemma~\ref{lem:simpleFlip}.
	
	Assume now that \(\bu_1 = \bu_2 b\) and \(\bv_1 = a\bv_2\). Let \((x_3,y_3)\) be the endpoint of \(\bu_2\) and \((x_4,y_4)\) be the one of \(\bu_2 b \overline{\bz} ab \bz a\). 
	Hence,
	\begin{align*}
		\delta_\bw(x_3,y_3)+\delta_\bw(x_4,y_4) &=  \frac{d(y_3+y_4)-n(x_3+x_4)}{d} 
		& =  \frac{2dy-2nx}{d} \\
		& = 2 \delta_\bw(x,y).
	\end{align*}
	By minimality of \(\delta_\bw(x,y)\), it would imply that \(\delta_\bw(x_3,y_3) = \delta_\bw(x_4,y_4)=\delta_\bw(x,y)\). But, as \(x_3<x\), it contradicts the minimality of the couple \((\delta_\bw(x,y),x)\).
	
	So we are in case 1 of Lemma~\ref{lem:simpleFlip}. Hence, as \(\bw = \bu^\prime ab \bv^\prime\),
	\begin{equation}\label{eq:lifting1}
		m(\bw) \geq m(\bu^\prime ba \bv^\prime).
	\end{equation}
	
	By construction, \(\bel(\bu^\prime ba \bv^\prime) = \bel(\bw) \cup \{(x-1,y+1)\}\) and \(\delta_\bw(x-1,y+1) = \delta_\bw(x,y) +1+n/d < 0\). So,
	\(\lift(\bu^\prime ba \bv^\prime) = \lift(\bw)\). Finally, by induction hypothesis, 
	\begin{equation}\label{eq:lifting2}
		m(\bu^\prime ba \bv^\prime) \geq m(\lift(\bw)).
	\end{equation}
	Combining Equations (\ref{eq:lifting1}) and (\ref{eq:lifting2}) concludes the induction and so proves the lemma.
\end{proof}

We are now ready to find the minimal words for the \(m\)-value (when \(n\) and \(d\) are fixed). It will identify the values \(m(c(n,d))\) to the values \(m_{n,d}\) defined in the Introduction. As stated before, we will need to consider this time composite flips instead of simple flips.

Let \((d,n)\in(\NN\setminus\{0\})^2\) with \(\delta = \gcd(d,n)\). We  recall that
\[
	c(n,d) = (\chris{n/d})^\delta
\]
and we know that \(\chris{n/d}\) is of the form \(a\bp_{n/d}b\) where \(\bp_{n/d}\) is a palindrome.

\begin{proposition}\label{prop:minimality}
	Let \((d,n)\in\NN^2\) with \(\delta = \gcd(d,n)\). Let \(\bw\) be a path which ends at \((d,n)\). Then,
		\[
			m(\bw) \geq m(c(n,d)).
		\]
	Moreover,  it is an equality if and only if 
	\[
		\bw = c(n,d) = (a\bp_{n/d}b)^\delta \text{ or } \begin{cases}
			\delta \geq 2 \\
			\bw = (a\bp_{n/d}a)(b\bp_{n/d}a)^{\delta-2} (b\bp_{n/d}b).
		\end{cases}
	\]
\end{proposition}

In particular, it can be noticed this also implies Conjecture 10.16 from~\cite{Aig15} (with our notations, this conjecture states that for \(n\leq d\), \(c(n,d)\) has minimal \(m\)-value among all paths which do not contain the factor `bb'). More information can be found in Appendix~\ref{app:gFunc}.

\begin{proof}
	Let \(\bw \in \{a,b\}^\star\) with \(d\) letters  `a' and \(n\) `b'. Let us consider \(\lift(\flat(\bw))\). By Lemmas~\ref{lem:flattening} and~\ref{lem:lifting}, we know
	\[
		m(\bw) \geq m(\lift(\flat(\bw))).
	\]
	During the flattening operation, we removed from \(\bel(\bw)\) every point which stand strictly above the line from \((0,0)\) to \((d,n)\). Then, during the lifting operation, we add all points which stands strictly below this line.
	
	Hence, apart from the intermediate points which are exactly on the line, the set \(\bel(\lift(\flat(\bw)))\) corresponds to the set \(T_{n,d}\). Consequently, in the case where \(\gcd(d,n)=1\), there are no integers point (except for the extremities) on the segment from \((0,0)\) to \((d,n)\), and so \(\bel(\lift(\flat(\bw))) = T_{n,d}\), which means that \(\lift(\flat(\bw)) = c(n,d)\) and the proposition is verified.
	
	So, we can assume now that \(\delta = \gcd(d,n) >1\). There are exactly \(\delta-1\) intermediate points which stand on the segment from \((0,0)\) to \((d,n)\): these points are \[
	(u_i,v_i) \defeq \left(i\frac{d}{\delta},i\frac{n}{\delta} \right)\ \text{ for } i \in \{1,2,\ldots,\delta-1\}.
	\] 
	In function if \(\bw\) goes through the point \((u_i,v_i)\) or not, \(\lift(\flat(\bw)) \) is of the form:
	\[
		a \bp_{n/d}\left\{\begin{array}{c} ba \\ ab \end{array}\right\} \bp_{n/d} \left\{\begin{array}{c} ba \\ ab \end{array}\right\}  \bp_{n/d} \cdots \bp_{n/d}\left\{\begin{array}{c} ba \\ ab \end{array}\right\} \bp_{n/d} b 
	\]
	where the notation \(\left\{\begin{array}{c} ba \\ ab \end{array}\right\}\) which appears \(\delta-1\) times means `\(ab\) or \(ba\)'.

	For simplicity of notation, let us define
	\[
		\bz_1 = ba \ \text{ and }\ \bz_{-1}=ab.
	\]
	Let \( \beps = (\varepsilon_i)_{i \in \{1,\ldots,\delta-1\}} \in \{\pm 1\}^{\delta-1}\). We consider the words
	\[
		\bu_{\beps} = a \bp_{n/d} \bz_{\varepsilon_1} \bp_{n/d} \bz_{\varepsilon_2} \bp_{n/d} \cdots \bp_{n/d} \bz_{\varepsilon_{\delta-1}} \bp_{n/d} b.
	\]
	In particular, if \(\beps\) is defined by \(\varepsilon_i = 1\) if and only if \(\bw\) goes through the point \((u_i,v_i)\), the word \(\bu_{\beps}\) is exactly the word \(\lift(\flat(\bw))\).
	
	Let us show that among all the words \(\bu_{\beps}\) for different values of \(\beps\), the \(m\)-value is minimal if \(\beps\) is a constant sequence.
	\begin{itemize}
		\item First, if for all \(i\), \(\varepsilon_i =1\), we have
		\[
		\bu_{\beps} = (a\bp_{n/d} b)^\delta = c(n,d).
		\]
		\item Then if for all \(i\), \(\varepsilon= -1\), we have
		\[
		\bu_{\beps} = (a\bp_{n/d} a) (b\bp_{n/d} a)^{\delta-2}(b\bp_{n/d} b) 
		\]
		which are the words which appear in the statement of the proposition. Their \(m\)-value equals the one of the corresponding Christoffel words: using the fourth point of Lemma~\ref{lem:compositeFlip}, we have
		\begin{align*}
			m((a\bp_{n/d} a) (b\bp_{n/d} a)^{\delta-2}(b\bp_{n/d} b)) & = m((a\bp_{n/d} b) (a\bp_{n/d} b)^{\delta-2}(a\bp_{n/d} b)) \\
			& = m(c(n,d))
		\end{align*}
		(one can notice that we could also use Lemma~\ref{lem:ab}).)
		\item Otherwise the number of sign changes in the sequence \((\varepsilon_i)\) is at least \(1\). Let us show by induction on this number of sign changes that \(m(\bu_{\beps}) > m(c(n,d))\). Let \(j \geq 2\) be the index of the first sign change, i.e., \(\varepsilon_1 = \ldots = \varepsilon_{j-1} = - \varepsilon_j\). So, \(\bu_{\beps}\) is of one of the two forms (depending on the sign of \(\varepsilon_1\))
		\begin{align*}
			& a \bp_{n/d} a (b \bp_{n/d} a)^{j-2} b \bp_{n/d} ba \bv \ \text{ where }\bv \in \{a,b\}^\star, \\
			\text{or, } & a \bp_{n/d} b (a \bp_{n/d} b)^{j-2} a \bp_{n/d} ab \bv \ \text{ where }\bv \in \{a,b\}^\star.
 		\end{align*}
 	For the first form, we can apply the first point of Lemma~\ref{lem:compositeFlip} on the factor \(a(b\bp_{n/d}a)^{j-2}b\). As the suffix \(a\bv\) is non empty, the inequality is strict. And for the second form, we can apply the third point of Lemma~\ref{lem:compositeFlip} on the factor \(b(a\bp_{n/d}b)^{j-2}a\). In both cases, we finish with the word \(\bu_{(-\varepsilon_1,\ldots,-\varepsilon_{j-1},\varepsilon_j,\ldots)}\) where the number of change of signs decreased by one. 
	\end{itemize}

\end{proof}

\subsection{Fixed sum conjecture}

Finally, we can also get the fixed sum conjecture:
\begin{proposition}\label{prop:FSC}
	Let \((d,n) \in \NN^2\) such that \(d \geq n+1\). We have \(m(c(n+1,d)) < m(c(n,d+1))\).
\end{proposition}

Let us start by proving the following claim:
\begin{claim}\label{clm:eq-fsc}
	Let \(n \in \NN\), then 
	\[
	 (AB)^nA - B(AB)^n  = \Pell_{2n+1} \begin{pmatrix} -1 & 0\\ 0 & 1
	\end{pmatrix} 
	\]
	where \(\Pell_k\) is the \(k^{\textrm{th}}\) Pell number.
\end{claim}

\begin{proof}
	In Equation~(\ref{eq:Pell}), we saw that
	\[
	(AB)^n = 
	\begin{pmatrix} 
		\Pell_{2n+1}-\Pell_{2n} & \Pell_{2n}\\ 2\Pell_{2n} & \Pell_{2n+1}-\Pell_{2n}
	\end{pmatrix}.
	\]
	Therefore we can conclude that for any \(n \in \NN\)
	\begin{align*}
		(AB)^nA - B(AB)^n &= \begin{pmatrix} 		
			\Pell_{2n+1}-\Pell_{2n} & \Pell_{2n}\\ 2\Pell_{2n} & \Pell_{2n+1}-\Pell_{2n}
		\end{pmatrix} \begin{pmatrix} 1 & 1\\ 1 & 2
		\end{pmatrix}  \\
	& \quad \quad	- 
	\begin{pmatrix} 2 & 1\\ 1 & 1
	\end{pmatrix} \begin{pmatrix} 		
		\Pell_{2n+1}-\Pell_{2n} & \Pell_{2n}\\ 2\Pell_{2n} & \Pell_{2n+1}-\Pell_{2n}
	\end{pmatrix}\\
		& = \begin{pmatrix} 		
			\Pell_{2n+1} & \Pell_{2n+1}+\Pell_{2n} \\
			\Pell_{2n+1}+\Pell_{2n} & 2\Pell_{2n+1}
		\end{pmatrix} \\
		& \quad \quad	- 
	\begin{pmatrix} 		
		2\Pell_{2n+1} & \Pell_{2n+1}+\Pell_{2n} \\ 
		\Pell_{2n+1}+\Pell_{2n} & \Pell_{2n+1}
	\end{pmatrix} \\
		&= \begin{pmatrix} -P_{2n+1} & 0\\ 0 & P_{2n+1}
		\end{pmatrix}.
	\end{align*}

\end{proof}

We can now prove the fixed sum conjecture.

\begin{proof}[Proof of Proposition~\ref{prop:FSC}]
		By hypothesis, \(d\geq 1\). If \(n=0\), then \(m(c(n+1,d)) = \Fibo_{2d+1} < \Fibo_{2d+2}=m(c(n,d+1))\).
		
		So we assume that \(n\geq 1\).
		Let us consider \(c(n,d+1)\). Since \(d+1 > n\), the word \(c(n,d+1)\) ends with `ab' and does not contain `bb' as a factor. 
	
	Let \(k \in \NN\) be maximal such that \((ab)^k ab\) is a suffix of \(c(n,d+1)\) (when \(k=0\), \((ab)^k ab = ab\) is really a suffix, so \(k\) is well defined). 
	
	We know that \(\lvert c(n,d+1)\rvert_a  - \lvert c(n,d+1)\rvert_b = d+1 - n \geq 2\). So \((ab)^k ab\) is a proper suffix. As `bb' is not a factor of \(c(n,d+1)\) and due to maximality of \(k\), the word \(c(n,d+1)\) has also a suffix \(a(ab)^k ab\) and this suffix is proper as well. So we can write \(c(n,d+1) = \bw a (ab)^k ab\) for a non-empty factor \(\bw\). By Claim~\ref{clm:eq-fsc},
	\begin{align*}
		m(\bw a (ab)^k ab)-m(\bw a b(ab)^k b) 
		& = \begin{pmatrix}
			1 & 0
		\end{pmatrix} M^\bw A \left( (AB)^kA - B(AB)^k \right) B \begin{pmatrix}
			0 \\1
		\end{pmatrix} \\
		& = \Pell_{2k+1} \begin{pmatrix}
			1 & 0
		\end{pmatrix} M^\bw  A \begin{pmatrix} -1&0 \\ 0&1\end{pmatrix} B \begin{pmatrix}
			0 \\1
		\end{pmatrix} \\
		& = \Pell_{2k+1} \begin{pmatrix}
			1 & 0
		\end{pmatrix} M^\bw  \begin{pmatrix} -1&0 \\ 0&1 \end{pmatrix} \begin{pmatrix}
			0 \\1
		\end{pmatrix} \\
		& >0.
	\end{align*}
	
	However, \(\bw ab(ab)^k b\) has same endpoint than \(c(n+1,d)\). So by Proposition~\ref{prop:minimality}, 
	\[
	m(\bw ab(ab)^k b) \geq m(c(n+1,d)).
	\]
	
\end{proof}

\appendix
\bibliographystyle{ieeetr}
\bibliography{Markov}

\appendix

\section{How to go from Cohn words in~\cite{Aig15} to Proposition~\ref{prop:cohnMarkov} in this paper}\label{app:CohnW}

We will write \(\QQ_{0,1}\) for \(\QQ \cap [0,1]\).

In~\cite{Aig15}, the author associates to each word from the alphabet \(\{\mathfrak{a},\mathfrak{b}\}^\star\) a matrix in \(\rm{SL}_2(\ZZ)\) given by the morphism \(\phi\) generated by \(\phi(\mathfrak{a}) = \mathfrak{A} = \begin{pmatrix} 1 & 1 \\ 1&2\end{pmatrix} = A\) and \(\phi(\mathfrak{b}) = \mathfrak{B} = \begin{pmatrix} 3 & 2 \\ 4 & 3\end{pmatrix} = AB\). In particular the Cohn tree is defined by starting from the triplet \((\mathfrak{A},\mathfrak{AB},\mathfrak{B})\) and by applying the following rule. Given any node associated to a triplet \((\mathfrak{C},\mathfrak{CD},\mathfrak{D})\), one introduces the two new children triplets (in the given order)  \((\mathfrak{C},\mathfrak{CCD},\mathfrak{CD})\) and \((\mathfrak{CD},\mathfrak{CDD},\mathfrak{D})\).

The Farey tree is defined similarly by starting from the triplet \((0/1,1/2,1/1)\) and by applying this following rule: given any node labeled by the triplet \((p_1/q_1,p_2/q_2,p_3/q_3)\) the two children are given by \((p_1/q_1,(p_1+p_2)/(q_1+q_2),p_2/q_2)\) and \((p_2/q_2,(p_2+p_3)/(q_2+q_3),p_3/q_3)\). It is proved that any fraction in \(\QQ_{0,1}\) appears once and exactly once as a second element of a triplet of the Farey tree.

For \(t \in \QQ_{0,1}\), they define \(C_t\) as the matrix which appears as the second element of the triplet which is at the same position in the Cohn tree than \(t\) is in the Farey tree. They define by \(W_t\) the word over the alphabet \(\{\mathfrak{a},\mathfrak{b}\}\) corresponding to the product of matrices associated to the matrix \(C_t\).

Let \(G\) be the word morphism given by 
\begin{align*}
	G : \{\mathfrak{a},\mathfrak{b}\}^\star & \rightarrow \{a,b\}^\star \\
		\mathfrak{a} & \mapsto a \\
		\mathfrak{b} & \mapsto ab. 
\end{align*}
So it implies that for \(\mathbb{\mathfrak{w}} \in \{\mathfrak{a},\mathfrak{b}\}^\star\), \(\phi(\mathbb{\mathfrak{w}}) = M^{G(\mathbb{\mathfrak{w}})}\). In particular for any \(t \in \QQ_{0,1}\), \(C_t = \phi(W_t) = M^{G(W_t)}\). As the set of Markov numbers is exactly the set \(\bigcup_{t \in \QQ_{0,1}} \left\{ \begin{pmatrix} 1&0\end{pmatrix} C_t \begin{pmatrix} 0 \\ 1\end{pmatrix}\right\}\), we can conclude that  the Markov numbers are exactly the values of \(m(G(W_t))\) when \(t\) goes through \(\QQ_{0,1}\).

For \(p\) and \(q\) relatively prime with \(p<q\), they define \(\rm{ch}_{p/q}\) as the lower Christoffel word of slope \(p/(q-p)\). Lemma 2.2 in~\cite{BLRS09} states that \(G(\rm{ch}_{p/q})\) is the Christoffel word of slope \(p/q\) that we called \(\mathfrak{c}_{p/q}\).

Theorem 7.6 in~\cite{Aig15} is
\begin{theorem*}
	Let \(t  \in \QQ_{0,1}\). Then \(W_t = \rm{ch}_{t}\).
\end{theorem*}

It directly implies Proposition~\ref{prop:cohnMarkov}
\begin{proposition*}
	The set of Markov numbers is the set\[
		\{m(\mathfrak{c}_t) \mid  t \in \QQ_{0,1}\}.
	\]
\end{proposition*}

\section{Links with the \(g\) function from Chapter 10  in~\cite{Aig15}}\label{app:gFunc}

Aigner introduced the three fixed parameters conjectures in Section 10.1 in his book. Following them, few pages present an approach to attack these conjectures using a particular \(g\)  function. As this approach shares similarities with ours, it can be interesting to clarify these connections. We keep notations from the previous section.

Let \(\psi\) be the morphism \(\{\mathfrak{a},\mathfrak{b}\}^\star \rightarrow \mathcal{M}_3(\ZZ)\) defined by
	\[
		\mathfrak{A}_1 \defeq \psi(\mathfrak{a}) = \begin{pmatrix}
			0 & 0 & 0 \\
			0 & 1 & 1 \\
			1 & 1 & 2
		\end{pmatrix}
	\textrm{ and }
	\mathfrak{B}_1 \defeq \psi(\mathfrak{b}) = \begin{pmatrix}
		1 & 1 & 2 \\
		0 & 0 & 0 \\
		2 & 3 & 5
	\end{pmatrix}.
	\]
	
The function \(g\) is defined for any word \(\mathbb{\mathfrak{w}}\) on the alphabet \(\{\mathfrak{a},\mathfrak{b}\}^\star\) by
\[
	g(\mathbb{\mathfrak{w}}) = \begin{pmatrix}
		2 & 3 & 5
	\end{pmatrix}
	\psi(\mathbb{\mathfrak{w}})
	\begin{pmatrix}
		0 \\ 0 \\ 1
	\end{pmatrix}.
\]

However, this matrix product \(\psi\) and this function \(g\) are just the usual product of Cohn matrices and its top-right entry but ``written in a different basis". Let us choose:
\[
P \defeq \begin{pmatrix}
	1 & 1 & 2 \\ 1 & 2 & 3
\end{pmatrix}, \quad
Q_{\mathfrak{a}} \defeq \begin{pmatrix}
	0 & 0 \\ -3 & 2 \\ 2 & -1
\end{pmatrix}, \textrm{ and }\ 
Q_{\mathfrak{b}} \defeq \begin{pmatrix}
	3 & -2 \\ 0 & 0 \\ -1 & 1
\end{pmatrix}.
\]

We have 
\[
	\mathfrak{A}_1 = Q_{\mathfrak{a}} \mathfrak{A} P, \quad
	\mathfrak{B}_1 = Q_{\mathfrak{b}} \mathfrak{B} P, \textrm{ and }\ 
	P Q_{\mathfrak{a}} = P Q_{\mathfrak{b}} = I.
\]

Consequently (where \(I_3\) is the identity matrix of size three), 
\[
	\psi(\mathbb{\mathfrak{w}}) = \begin{cases}
		I_3 & \textrm{ if }\mathbb{\mathfrak{w}}\textrm{ is the empty word,} \\
		Q_{\mathfrak{a}} \phi(\mathbb{\mathfrak{w}}) P & \textrm{ if }\mathbb{\mathfrak{w}}\textrm{ starts with an `}a\text{',} \\
		Q_{\mathfrak{b}} \phi(\mathbb{\mathfrak{w}}) P & \textrm{ otherwise.} \\
	\end{cases}
\]
Then,
\begin{align*}
	g(\mathbb{\mathfrak{w}}) & = \begin{pmatrix}
		2 & 3 & 5
	\end{pmatrix}
	\psi(\mathbb{\mathfrak{w}})
	\begin{pmatrix}
		0 \\ 0 \\ 1
	\end{pmatrix} \\
		& = \begin{pmatrix}
		1 & 1
		\end{pmatrix}
		\phi(\mathbb{\mathfrak{w}})
		\begin{pmatrix}
			2 \\ 3
		\end{pmatrix} \\
	& = \begin{pmatrix}
		1 & 0
	\end{pmatrix}
	\mathfrak{A} \phi(\mathbb{\mathfrak{w}}) \mathfrak{B}
	\begin{pmatrix}
		0 \\ 1
	\end{pmatrix} \\
	& = \begin{pmatrix}
		1 & 0
	\end{pmatrix}
	\phi(\mathbb{\mathfrak{a}\mathfrak{w}}\mathfrak{b}) 
	\begin{pmatrix}
		0 \\ 1
	\end{pmatrix}.
\end{align*}
So \(g(\mathbb{\mathfrak{w}})\) is just the top-right entry of the Cohn matrix associated to the word \(\mathfrak{a} \mathbb{\mathfrak{w}} \mathfrak{b}\).

By the previous section, it implies that \(g(\mathbb{\mathfrak{w}}) = m(G(\mathfrak{a} \mathbb{\mathfrak{w}} \mathfrak{b})) = m(aG(\mathbb{\mathfrak{w}})ab)\). Using our machinery, Proposition 10.14 in~\cite{Aig15} is a direct corollary of Lemma~\ref{lem:ab}:
\[
	g(\overline{\mathbb{\mathfrak{w}}}) = m\left(aG(\overline{\mathbb{\mathfrak{w}}})ab\right) = m\left(a(a\overline{G(\mathbb{\mathfrak{w}})}a^{-1})ab\right) 
	= m\left(a G(\mathbb{\mathfrak{w}}) ab\right)=g(\mathbb{\mathfrak{w}}) .
\]
Moreover the fact that among the words with \(k\) `\(\mathfrak{a}\)' and \(l\) `\(\mathfrak{b}\)' \(g\) is maximal for \(\mathfrak{a}^k\mathfrak{b}^l\) and \(\mathfrak{b}^l\mathfrak{a}^k\) follows from the two inequalities (with \(k_1,k_2,l_1,l_2 \geq 1\))
\begin{align*}
	g(\mathfrak{a}^{k_1}\mathfrak{b}^{l_1}\mathfrak{a}^{k_2}\mathbb{\mathfrak{w}}) & = m(aa^{k_1}(ab)^{l_1}a^{k_2}G(\mathbb{\mathfrak{w}})ab) \\
	& < m(aa(ba)^{l_1}a^{k_1-1}a^{k_2}G(\mathbb{\mathfrak{w}})ab) = g(\mathfrak{b}^{l_1}\mathfrak{a}^{k_1+k_2}\mathbb{\mathfrak{w}})
\end{align*}
where we flipped the factor \(a^{k_1-1}(ab)^{l_1}\) (point 3 of Lemma~\ref{lem:compositeFlip}) and
\begin{align*}
	g(\mathfrak{b}^{l_1}\mathfrak{a}^{k_1}\mathfrak{b}^{l_2}\mathbb{\mathfrak{w}}) & = m(a(ab)^{l_1}a^{k_1}(ab)^{l_2}G(\mathbb{\mathfrak{w}})ab) \\
	& < m(aaa^{k_1}(ba)^{l_1-1}b(ab)^{l_2}G(\mathbb{\mathfrak{w}})ab) = g(\mathfrak{a}^{k_1}\mathfrak{b}^{l_1+l_2}\mathbb{\mathfrak{w}})
\end{align*}
where we flipped the factor \(b(ab)^{l_1-1}a^{k_1}\) (point 1 of Lemma~\ref{lem:compositeFlip}).

On the other direction, we could be interested in words with \(k\) `\(\mathfrak{a}\)' and \(l\) `\(\mathfrak{b}\)' which minimize \(g\). So it is sufficient to find a word \(a\bw ab \) with \(k+l+2\) `a', \(l+1\) `b' and without `bb' which minimizes \(m\). By Proposition~\ref{prop:minimality}, we know this is the case for Christoffel words which proves Conjecture 10.16 in~\cite{Aig15}.

Unfortunately, possibilities 2,3, and 4 from page 223 in this book are not true. Indeed, the possibility 2 claims that the \(g\) measure decreases when the number of changes of letters increases. However,
\[
	g(\mathfrak{a}^2\mathfrak{b}\mathfrak{a}^3\mathfrak{b}^2\mathfrak{a}^3\mathfrak{b}\mathfrak{a}^3) = 210098378 < 210106196 = g\left(\mathfrak{a}^7(\mathfrak{ab})^4\right) 
\]
whereas the first word has \(6\) changes of letters and the second word \(7\). The possibility 3 would be a great strengthening of Frobenius' conjecture: it states that if \(g(\mathbb{\mathfrak{w}}_1) = g(\mathbb{\mathfrak{w}}_2) \), then \(\mathbb{\mathfrak{w}}_1 = \mathbb{\mathfrak{w}}_2\) or \(\mathbb{\mathfrak{w}}_1 = \overline{\mathbb{\mathfrak{w}}_2}\). But
\[
	g(\mathfrak{a}\mathfrak{b}\mathfrak{b}\mathfrak{a}) = 1130 = g(\mathfrak{b}\mathfrak{a}\mathfrak{a}\mathfrak{b}).
\]
Finally, the possibility 4 suggests that the ratio \(g\left(\mathfrak{a}^k\mathfrak{b}^l) / g(\mathfrak{a}^{-1}\textrm{ch}_{(l+1)/(k+l+2)}\mathfrak{b}^{-1}\right)\) is bounded near to \(1.1\). Computations show this is not true at all, for example for \(k=400\) and \(l=100\), the ratio is \(\approx 12\) and for \(k= 3000\) and \(l= 1000\) the ratio becomes \(\approx 8\cdot10^{10}\).

Similarly Conjectures 10.18 at page 225 can be falsified by computations:
\begin{align*}
	g(\mathfrak{a}^{46}\mathfrak{b}^{45}) > g(\mathfrak{a}^{-1}\textrm{ch}_{46/94}\mathfrak{b}^{-1}) & \textrm{ and }\mathfrak{a}^{-1}\textrm{ch}_{46/94}\mathfrak{b}^{-1}\textrm{ has }47\textrm{ `}\mathfrak{a}\textrm{' and }45\textrm{ `}\mathfrak{b}\textrm{',} \\
	g(\mathfrak{a}^{40}\mathfrak{b}^{37}) > g(\mathfrak{a}^{-1}\textrm{ch}_{39/79}\mathfrak{b}^{-1}) & \textrm{ and }\mathfrak{a}^{-1}\textrm{ch}_{39/79}\mathfrak{b}^{-1}\textrm{ has }39\textrm{ `}\mathfrak{a}\textrm{' and }38\textrm{ `}\mathfrak{b}\textrm{',} \\
	g(\mathfrak{a}^{8}\mathfrak{b}^{9}) > g(\mathfrak{a}^{-1}\textrm{ch}_{9/20}\mathfrak{b}^{-1}) & \textrm{ and }\mathfrak{a}^{-1}\textrm{ch}_{9/20}\mathfrak{b}^{-1}\textrm{ has }10\textrm{ `}\mathfrak{a}\textrm{' and }8\textrm{ `}\mathfrak{b}\textrm{'.}
\end{align*}

\end{document}